\documentclass[a4paper,12pt]{article}
\usepackage{amsmath}
\usepackage{amssymb}	
\usepackage{amsthm}\usepackage{latexsym}
\usepackage{amsmath}

\usepackage{amsthm}
\usepackage[dvipdfmx]{graphicx}
\usepackage{txfonts}
\usepackage{bm}
\usepackage{color}
\usepackage{epic,eepic}

\usepackage{rotating}

\usepackage{geometry}
\numberwithin{equation}{section}

\newtheorem{theorem}{Theorem}[section]
\newtheorem{proposition}[theorem]{Proposition}
\newtheorem{corollary}[theorem]{Corollary}

\newtheorem{remark}{Remark}[section]
\newtheorem{example}{Example}[section]

\newcommand{\OMIT}[1]{{\bf [OMIT:} #1 \ {\bf --- end OMIT] }}  
   \renewcommand{\OMIT}[1]{}            

\newcommand{\RR}{{\mathbb{R}}}
\newcommand{\ZZ}{{\mathbb{Z}}}
\newcommand{\vecone}{{\bf 1}}
\newcommand{\veczero}{{\bf 0}}
\newcommand{\dom}{{\rm dom\,}}
\newcommand{\domZ}{{\rm dom\,}}

\newcommand{\subg}{\partial}

\newcommand{\suppp}{{\rm supp}\sp{+}}
\newcommand{\suppm}{{\rm supp}\sp{-}}

\newcommand{\unitvec}[1]{\bm{1}\sp{#1}}

\newcommand{\argmax}{\arg \max}
\newcommand{\conv}{\Box\,}

\newcommand{\finbox}{\hspace*{\fill}$\rule{0.2cm}{0.2cm}$}
\newcommand{\todaye}{\the\year/\the\month/\the\day}

\newcommand{\Lnat}{{L$^{\natural}$}}
\newcommand{\Mnat}{{M$^{\natural}$}}

\newcommand{\Bvex}{\mbox{\rm (B-EXC)} }
\newcommand{\Bvexb}{\mbox{\rm\bf (B-EXC)}}
\newcommand{\Bnvex}{\mbox{\rm (B$\sp{\natural}$-EXC)} }
\newcommand{\Bnvexb}{\mbox{\rm\bf (B$\sp{\natural}$-EXC)}}

\newcommand{\Mvexb}{\mbox{\rm\bf (M-EXC)}}
\newcommand{\Mnvex}{\mbox{\rm (M$\sp{\natural}$-EXC)} }
\newcommand{\Mnvexb}{\mbox{\rm\bf (M$\sp{\natural}$-EXC)}}

\newcommand{\YES}{Y \ }
\newcommand{\NO}{\quad \textbf{\textit{N}}}

\begin{document}

\title{A Survey of Fundamental Operations on \\  
Discrete Convex Functions of Various Kinds%
}

\author{
Kazuo Murota%
\thanks{Department of Economics and Business Administration,
Tokyo Metropolitan University, 
Tokyo 192-0397, Japan, 
murota@tmu.ac.jp}
}

\date{July 2019 / September 2019 / October 2019}

\maketitle

\begin{abstract}
Discrete convex functions are used in many areas,
including operations research, discrete-event systems, game theory, and economics.
The objective of this paper is to offer a survey on fundamental operations
for various kinds of discrete convex functions in discrete convex analysis
such as
integrally convex functions,
{\rm L}-convex functions, {\rm M}-convex functions,
and multimodular functions.
\end{abstract}

{\bf Keywords}:
Discrete convex analysis,  Integrally convex function,
{\rm L}-convex function, {\rm M}-convex function, 
Multimodular function, 
Submodular function, Valuated matroid

\newpage
\tableofcontents
\newpage



\section{Introduction}
\label{SCintro}

Discrete convex functions are used in many areas,
including operations research, discrete-event systems, game theory, and economics
\cite{Fuj05book,Mdca98,Mdcasiam,Mbonn09,Mdcaeco16,ST15jorsj,SCB14}.
The objective of this paper is to offer a survey of 
fundamental operations of various kinds of discrete convex functions in discrete convex analysis.

Discrete convex functions treated in this paper
include
integrally convex functions \cite{FT90,MMTT19proxIC},
{\rm L}- and \Lnat-convex functions
\cite{FM00,Mdca98,Mdcasiam},
{\rm M}- and \Mnat-convex functions
\cite{Mstein96,Mdca98,Mdcasiam,MS99gp},
multimodular functions
\cite{AGH00,AGH03,GY94mono,Haj85},
globally and locally discrete midpoint convex functions 
\cite{MMTT19dmpc},
and M- and \Mnat-convex functions on jump systems
\cite{KMT07jump,Mmjump06,Mmnatjump19}.
L-convex functions on trees and graphs 
\cite{Hir15Lext,Hir16min0ext,Hir17dcfgrBK,Hir18Lgraph,Kol11}
are outside the scope of this paper.
It is worth noting that 
``L'' stands for ``Lattice'' and ``M'' for ``Matroid.''  
It is also noted that
``\Lnat'' and ``\Mnat''  should be pronounced as ``ell natural'' 
 and  ``em natural,'' respectively.

Various operations can be defined for discrete functions  
$f: \ZZ\sp{n} \to \RR \cup \{ +\infty  \}$.
With changes of variables 
we can define operations such as 
origin shift  $f(x) \mapsto f(x-b)$,
coordinate inversion $f(x) \mapsto f(-x)$, 
permutation of variables 
 $f(x) \mapsto f(x_{\sigma(1)}, x_{\sigma(2)}, \ldots, x_{\sigma(n)})$,
and scaling of variables $f(x) \mapsto f(\alpha x)$
with a positive integer $\alpha$.
With arithmetic or numerical operations on function values
we can define 
nonnegative multiplication of function values $f(x) \mapsto a f(x)$ with $a \geq 0$,
addition of a linear function
$f(x) \mapsto f(x) + \sum_{i=1}\sp{n} c_{i} x_{i}$
with  $c \in \RR\sp{n}$, projection
(partial minimization) $f(x) \mapsto  \inf_{z} f(y,z)$,
sum $f_{1} + f_{2}$ of two functions $f_{1}$ and $f_{2}$,
convolution 
$(f_{1} \conv f_{2})(x) 
= \inf\{ f_{1}(y) + f_{2}(z) \mid x= y + z, \  y, z \in \ZZ\sp{n} \}$
of two functions $f_{1}$ and $f_{2}$,
etc.
Stability of discrete convexity under these operations
has been investigated for many function classes 
in discrete convex analysis
\cite{KMT07jump,MM19projcnvl,MMTT19proxIC,MMTT19dmpc,Mdcasiam,MS01rel}.
By collecting these results scattered in the literature
and by adding new observations and examples,
we will present a fairly comprehensive survey 
on the operations on discrete convex sets and functions.

This paper is organized as follows.
Section~\ref{SCfnclass} is a review of the definitions of discrete convex
sets and functions.
Section~\ref{SCsetope} treats operations on discrete convex sets
such as restriction, projection, and Minkowski sum.
Section~\ref{SCfnope} treats operations on discrete convex functions
such as restriction, projection, convolution, and
discrete Legendre--Fenchel transformation.

The classes of discrete convex sets and functions considered in this paper
are listed in Tables  \ref{TBdefdiscconvset} and \ref{TBdefdiscconvfn}  
with brief descriptions of their definitions,
while the precise definitions are given in Section~\ref{SCfnclass}.

\begin{table}[h]
\caption{Various kinds of discrete convex sets}
\label{TBdefdiscconvset}
\begin{center}
\begin{tabular}{l|l}
 Convex set &  Defining condition (roughly)
\\ \hline
 Integer box & 
  $[a,b]_{\ZZ} = \{ x \in \ZZ\sp{n} \mid a \leq x  \leq b \}$ 
\\ 
 Integrally convex &
  The union of local convex hulls is convex
\\ 
\Lnat-convex &  
$x, y \in S \ \Rightarrow \ 
  \left\lceil \frac{x+y}{2} \right\rceil, \left\lfloor \frac{x+y}{2} \right\rfloor \in S $
\\ 
L-convex  & \Lnat-convex \& invariance in direction $\vecone$
\\ 
\Mnat-convex  & 
$x, y \in S \ \Rightarrow \ 
 x - \unitvec{i} + \unitvec{j}, y + \unitvec{i} - \unitvec{j} \in S$ 
\\ 
M-convex   &  \Mnat-convex \& constant component-sum 
\\ 
Multimodular  & 
$x+d, x+d' \in S,
d=\unitvec{i}-\unitvec{i+1}, 
d'=\unitvec{j}-\unitvec{j+1}$ 
\\ & 
 $(0 \leq i < j \leq n)
 \ \Rightarrow \ x, x+d+d'\in S$
\\ 
Disc.~midpt convex
& 
$x, y \in S, \| x - y \|_{\infty} \geq 2 \ \Rightarrow \ 
  \left\lceil \frac{x+y}{2} \right\rceil, \left\lfloor \frac{x+y}{2} \right\rfloor \in S $
\\ 
Simul.~exch. jump & 
$x, y \in S \ \Rightarrow \ 
 x \pm \unitvec{i} \pm \unitvec{j}, y \mp \unitvec{i} \mp \unitvec{j} \in S$ 
\\ 
Const-parity jump & 
Simul.~exch. jump \& constant-parity
\\ \hline
\end{tabular}
\end{center}
\end{table}

\begin{table}[h]
\caption{Various kinds of discrete convex functions}
\label{TBdefdiscconvfn}
\begin{center}
\begin{tabular}{l|c|l}
 Convex function & Dom & Defining condition (roughly)
\\ \hline
 Submod.~set fn  & $2\sp{N}$ & $f(X) +  f(Y) \geq f(X \cup Y) + f(X \cap Y)$
\\ 
 Valuated matroid & $2\sp{N}$ & $f(X) +  f(Y) \leq  \max \{ f(X-i+j) + f(Y+i-j) \}$
\\ \hline
 Separable convex & $\ZZ\sp{n}$ & 
 $f(x) = \varphi_{1}(x_1) + \varphi_{2}(x_2) + \cdots + \varphi_{n}(x_n)$ 
($\varphi_{i}$: convex)
\\ 
 Integrally convex & $\ZZ\sp{n}$ &  
  Local convex extension is (globally) convex
\\ 
\Lnat-convex &  $\ZZ\sp{n}$ &
$ f(x) + f(y) \ \geq \ 
   f \left(\left\lceil \frac{x+y}{2} \right\rceil\right) 
  + f \left(\left\lfloor \frac{x+y}{2} \right\rfloor\right) $
\\ 
L-convex  & $\ZZ\sp{n}$  &  \Lnat-convex \& linear in direction $\vecone$ 
\\ 
\Mnat-convex  & $\ZZ\sp{n}$ & 
$\displaystyle  f(x) + f(y) 
 \geq  \min \{f(x - \unitvec{i} + \unitvec{j}) + f(y + \unitvec{i} - \unitvec{j}) \}$ 
\\ 
M-convex  & $\ZZ\sp{n}$  &  \Mnat-convex \& constant sum of $\dom f$ 
\\ 
Multimodular  & $\ZZ\sp{n}$ & 
$f(x) = g(x_{1},x_{1}+x_{2},\ldots,x_{1}+ \cdots + x_{n})$, $g$: \Lnat-convex
\\ 
Globally d.m.c. 
&  $\ZZ\sp{n}$ &
$ f(x) + f(y) \ \geq \ 
   f \left(\left\lceil \frac{x+y}{2} \right\rceil\right) 
  + f \left(\left\lfloor \frac{x+y}{2} \right\rfloor\right) $
$( \| x - y \|_{\infty} \geq 2)$
\\ 
Locally d.m.c. 
&  $\ZZ\sp{n}$ &
$ f(x) + f(y) \ \geq \ 
   f \left(\left\lceil \frac{x+y}{2} \right\rceil\right) 
  + f \left(\left\lfloor \frac{x+y}{2} \right\rfloor\right) $
$( \| x - y \|_{\infty} = 2)$
\\ 
Jump \Mnat-convex  & $\ZZ\sp{n}$  & 
$\displaystyle  f(x) + f(y)  
 \geq 
  \min\{  f(x \pm \unitvec{i} \pm \unitvec{j}) + f(y \mp \unitvec{i} \mp \unitvec{j}) \}$ 
\\ 
Jump M-convex  & $\ZZ\sp{n}$  & 
     Jump \Mnat-convex  \& constant-parity of $\dom f$
\\ \hline
\multicolumn{3}{l}{${}\sp{*}$ A valuated matroid is discrete concave.} 
\end{tabular}
\end{center}
\end{table}

\paragraph*{Notations}
We use the following notations.

\begin{itemize}
\item
The set of all real numbers is denoted by $\RR$, and
the set of all integers is denoted by $\ZZ$.

\item 
We assume  $N = \{ 1,2, \ldots, n \}$
for a positive integer $n$.

\item
The characteristic vector of a subset
$A \subseteq N= \{ 1,2, \ldots, n \}$
is denoted by $\unitvec{A}$,   
that is,
\begin{equation} \label{charvecdefnotat}
 (\unitvec{A})_{i} =
   \left\{  \begin{array}{ll}
     1      & (i \in A) , \\
     0      & (i \in N \setminus A). \\
                      \end{array}  \right.
\end{equation}
For $i \in \{ 1,2, \ldots, n \}$,
we write $\unitvec{i}$ for $\unitvec{ \{ i \} }$, 
which is the $i$th unit vector.
We define $\unitvec{0}=\veczero$
where
$\veczero =(0,0,\ldots,0)$. 
We also define $\vecone=(1,1,\ldots,1)$.

\item
For a vector $x=(x_{1},x_{2}, \ldots, x_{n})$
 and a subset $A \subseteq \{ 1,2, \ldots, n \}$,
$x(A)$ denotes the component sum within $A$,
i.e.,
$x(A) = \sum \{ x_{i} \mid i \in A \}$.

\item
For two vectors 
$x=(x_{1},x_{2}, \ldots, x_{n})$
and  $y=(y_{1},y_{2}, \ldots, y_{n})$,
$x \leq y$ means that
$x_{i} \leq y_{i}$ for all $i=1,2,\ldots,n$.

\item 

For a vector $x=(x_{1},x_{2}, \ldots, x_{n})$
the {\em positive} and {\em negative supports} of $x$ are defined as
\begin{equation} \label{vecsupportdef}
 \suppp(x) = \{ i \mid x_{i} > 0 \},
\qquad 
 \suppm(x) = \{ i \mid x_{i} < 0 \}.
\end{equation}

\item 
The $1$-norm of a vector $x$
is denoted as $\| x \|_{1}$, i.e.,
$\| x \|_{1} = |x_{1}| +  |x_{2}| + \cdots + |x_{n}|$.

\item 
The $\ell_{\infty}$-norm of a vector $x$
is denoted as $\| x \|_{\infty}$, i.e.,
$\| x \|_{\infty} = \max( |x_{1}|, |x_{2}|, \ldots, |x_{n}| )$.

\end{itemize}





\section{Definitions of Discrete Convex Sets and Functions}
\label{SCfnclass}

We consider functions defined on integer lattice points, 
$f: \ZZ\sp{n} \to \RR \cup \{ +\infty \}$,
where the function may possibly take $+\infty$.
The {\em effective domain} of $f$ means the set of $x$
with $f(x) <  +\infty$ and is denoted 
by $\dom f =   \{ x \in \ZZ\sp{n} \mid  f(x) < +\infty \}$.
We always assume that $\dom f$ is nonempty.
The {\em indicator function} of a set $S \subseteq \ZZ\sp{n}$
is the function 
$\delta_{S}: \ZZ\sp{n} \to \{ 0, +\infty \}$
defined by
\begin{equation}  \label{indicatordef}
\delta_{S}(x)  =
   \left\{  \begin{array}{ll}
    0            &   (x \in S) ,      \\
   + \infty      &   (x \not\in S) . \\
                      \end{array}  \right.
\end{equation}
The convex hull of a set $S$ is denoted by $\overline{S}$.
A set $S$ is said to be {\em hole-free} if
\begin{equation}  \label{holefree}
S =  \overline{S} \cap \ZZ^{n},
\end{equation}
which means  that all integer points contained in the convex hull of $S$ are members of $S$.

\subsection{Separable convexity}

For integer vectors 
$a \in (\ZZ \cup \{ -\infty \})\sp{n}$ and 
$b \in (\ZZ \cup \{ +\infty \})\sp{n}$ 
with $a \leq b$,
$[a,b]_{\ZZ}$ denotes the integer box  
(discrete rectangle, integer interval)
between $a$ and $b$,
i.e.,
\begin{equation} \label{intboxdef}
[a,b]_{\ZZ} = \{ x \in \ZZ\sp{n} \mid a_{i} \leq x_{i} \leq b_{i} \ (i=1,2,\ldots,n)  \}.
\end{equation}

A function
$f: \ZZ^{n} \to \RR \cup \{ +\infty \}$
in $x=(x_{1}, x_{2}, \ldots,x_{n}) \in \ZZ^{n}$
is called  {\em separable convex}
if it can be represented as
\begin{equation}  \label{sepfndef}
f(x) = \varphi_{1}(x_{1}) + \varphi_{2}(x_{2}) + \cdots + \varphi_{n}(x_{n})
\end{equation}
with univariate functions
$\varphi_{i}: \ZZ \to \RR \cup \{ +\infty \}$ satisfying 
\begin{equation}  \label{univarconvdef}
\varphi_{i}(t-1) + \varphi_{i}(t+1) \geq 2 \varphi_{i}(t)
\qquad (t \in \ZZ).
\end{equation}

\subsection{Integrally convexity}

In this section we introduce the concept of integrally convex functions.

For $x \in \RR^{n}$ the integral neighborhood of $x$ is defined as 
\begin{equation}  \label{intneighbordef}
N(x) = \{ z \in \ZZ^{n} \mid | x_{i} - z_{i} | < 1 \ (i=1,2,\ldots,n)  \}.
\end{equation}
For a set $S \subseteq \ZZ^{n}$
and $x \in \RR^{n}$
we call the convex hull of $S \cap N(x)$ 
the {\em local convex hull} of $S$ at $x$.
A nonempty set $S \subseteq \ZZ^{n}$ is said to be 
{\em integrally convex} if
the union of the local convex hulls $\overline{S \cap N(x)}$ over $x \in \RR^{n}$ 
is convex \cite{Mdcasiam}.
This is equivalent to saying that,
for any $x \in \RR^{n}$, 
$x \in \overline{S} $ implies $x \in  \overline{S \cap N(x)}$.
An integrally convex set $S$ is hole-free in the sense 
of \eqref{holefree}.

For a function
$f: \ZZ^{n} \to \RR \cup \{ +\infty  \}$
the {\em local convex extension} 
$\tilde{f}: \RR^{n} \to \RR \cup \{ +\infty \}$
of $f$ is defined 
as the union of all convex envelopes of $f$ on $N(x)$.  That is,
\begin{equation} \label{fnconvclosureloc2}
 \tilde f(x) = 
  \min\{ \sum_{y \in N(x)} \lambda_{y} f(y) \mid
      \sum_{y \in N(x)} \lambda_{y} y = x,  \ 
  (\lambda_{y})  \in \Lambda(x) \}
\quad (x \in \RR^{n}) ,
\end{equation} 
where $\Lambda(x)$ denotes the set of coefficients for convex combinations indexed by $N(x)$:
\[ 
  \Lambda(x) = \{ (\lambda_{y} \mid y \in N(x) ) \mid 
      \sum_{y \in N(x)} \lambda_{y} = 1, 
      \lambda_{y} \geq 0 \ \ \mbox{for all } \   y \in N(x)  \} .
\] 
If $\tilde f$ is convex on $\RR^{n}$,
then $f$ is said to be {\em integrally convex}
\cite{FT90}.
The effective domain of an integrally convex function is an integrally convex set.
A set $S \subseteq \ZZ\sp{n}$ is integrally convex if and only if its indicator function
$\delta_{S}: \ZZ\sp{n} \to \{ 0, +\infty \}$
is an integrally convex function.

Integral convexity of a function can be characterized by a local condition
under the assumption that the effective domain is an integrally convex set.

\begin{theorem}[\cite{FT90,MMTT19proxIC}]
\label{THfavtarProp33}
Let $f: \ZZ^{n} \to \RR \cup \{ +\infty  \}$
be a function with an integrally convex effective domain.
Then the following properties are equivalent:


{\rm (a)}
$f$ is integrally convex.

{\rm (b)}
For every $x, y \in \ZZ\sp{n}$ with $\| x - y \|_{\infty} =2$  we have \ 
\begin{equation}  \label{intcnvconddist2}
\tilde{f}\, \bigg(\frac{x + y}{2} \bigg) 
\leq \frac{1}{2} (f(x) + f(y)).
\end{equation}
\vspace{-1.7\baselineskip} \\
\finbox
\end{theorem}

The reader is referred to \cite{FT90},
\cite[Section 3.4]{Mdcasiam}, and 
\cite[Section 13]{Mdcaeco16}
for more about integral convexity,
and \cite{MM19projcnvl} and \cite{MMTT19proxIC}
for recent developments.

\subsection{L-convexity}

\subsubsection{\Lnat-convex sets and functions}

A nonempty set $S \subseteq  \ZZ\sp{n}$ is called {\em \Lnat-convex} if
\begin{equation} \label{midptcnvset}
 x, y \in S
\ \Longrightarrow \
\left\lceil \frac{x+y}{2} \right\rceil ,
\left\lfloor \frac{x+y}{2} \right\rfloor  \in S ,
\end{equation}
where, for $t \in \RR$ in general, 
$\left\lceil  t   \right\rceil$ 
denotes the smallest integer not smaller than $t$
(rounding-up to the nearest integer)
and $\left\lfloor  t  \right\rfloor$
the largest integer not larger than $t$
(rounding-down to the nearest integer),
and this operation is extended to a vector
by componentwise applications.
The property \eqref{midptcnvset} is called {\em discrete midpoint convexity}.

A function $f : \ZZ\sp{n} \to \RR \cup \{ +\infty \}$ with $\dom f \not= \emptyset$
is said to be {\em L$\sp{\natural}$-convex}
if it satisfies a quantitative version of discrete midpoint convexity,
i.e., if 
\begin{equation} \label{midptcnvfn}
 f(x) + f(y) \geq
   f \left(\left\lceil \frac{x+y}{2} \right\rceil\right) 
  + f \left(\left\lfloor \frac{x+y}{2} \right\rfloor\right) 
\end{equation}
holds for all $x, y \in \ZZ\sp{n}$.
The effective domain of an \Lnat-convex function is an \Lnat-convex set.
A set $S$ is \Lnat-convex if and only if its indicator function
$\delta_{S}$ is an \Lnat-convex function.

It is known \cite[Section 7.1]{Mdcasiam} that
${\rm L}^{\natural}$-convex functions can be characterized 
by several different conditions,
stated in Theorem \ref{THlnatfncond} below.
The condition (b) in Theorem \ref{THlnatfncond}
imposes discrete midpoint convexity (\ref{midptcnvfn}) 
for all points $x,y$ at $\ell_\infty$-distance 1 or 2.
The condition (c) refers to 
{\em submodularity}, which means that
\begin{equation} \label{submfn}
f(x) + f(y) \geq f(x \vee y) + f(x \wedge y)
\end{equation}
holds for all $x, y \in \ZZ\sp{n}$, 
where
$x \vee y$ and $x \wedge y$ denote,
respectively, the vectors of componentwise maximum and minimum of $x$ and $y$, 
i.e.,
\begin{equation} \label{veewedgedef}
  (x \vee y)_{i} = \max(x_{i}, y_{i}),
\quad
  (x \wedge y)_{i} = \min(x_{i}, y_{i})
\qquad (i =1,2,\ldots, n).
\end{equation}
The condition (d) refers to 
a generalization of submodularity called
{\em translation-submodularity}, which means that
\begin{equation} \label{trsubmfn}
  f(x) + f(y) \geq f((x - \mu {\bf 1}) \vee y)
                 + f(x \wedge (y + \mu {\bf 1}))
\end{equation}
holds for all $x, y \in \ZZ^{n}$ and nonnegative integers $\mu$,
where $\bm{1}=(1,1,\ldots, 1)$.
The condition (e) refers to the condition%
\footnote{
This condition \eqref{lnatfnAPR} is labeled as  (L$\sp{\natural}$-APR[$\ZZ$]) in \cite[Section 7.2]{Mdcasiam}.
} 
 that, for any $x, y \in \ZZ\sp{n}$ with $\suppp(x-y)\not= \emptyset$, 
the inequality
\begin{equation} \label{lnatfnAPR}
 f(x) + f(y) \geq f(x - \unitvec{A}) + f(y + \unitvec{A}) 
\end{equation}
holds with $\displaystyle A = \argmax_{i} \{ x_{i} - y_{i} \}$,
where $\unitvec{A}$ denotes the characteristic vector of $A$.
The condition (f) refers to submodularity of
the function 
\begin{equation}\label{lfnlnatfnrelation}
 \tilde f(x_{0},x) = f(x - x_{0} \vecone)
 \qquad ( x_{0} \in \ZZ, x \in \ZZ\sp{n})
\end{equation}
in $n+1$ variables associated with the given function $f$.

\begin{theorem} \label{THlnatfncond}
For $f: \ZZ^{n} \to \RR \cup \{ +\infty \}$
the following conditions, {\rm (a)} to {\rm (f)}, are equivalent:


{\rm (a)}
$f$ is an \Lnat-convex function, that is,
it satisfies 
discrete midpoint inequality
{\rm (\ref{midptcnvfn})} for all $ x, y \in \ZZ^{n}$.

{\rm (b)}
$\dom f$ is an \Lnat-convex set,
and $f$ satisfies
discrete midpoint inequality
{\rm (\ref{midptcnvfn})} for all $ x, y \in \ZZ^{n}$
with $\| x-y \|_{\infty} \leq 2$.

{\rm (c)}
$f$ is integrally convex and submodular \eqref{submfn}.

{\rm (d)}
$f$ satisfies translation-submodularity \eqref{trsubmfn}
for all nonnegative $\mu \in \ZZ$.

{\rm (e)}
$f$ satisfies the condition \eqref{lnatfnAPR}.

{\rm (f)}
$\tilde f$ in \eqref{lfnlnatfnrelation} is submodular \eqref{submfn}.
\finbox
\end{theorem}

For a set $S \subseteq  \ZZ\sp{n}$ we consider conditions
\begin{align} 
 x, y \in S
& \ \Longrightarrow \
 x \vee y, \ x \wedge y \in S ,
\label{submset}
\\
 x, y \in S
& \ \Longrightarrow \
 (x - \mu {\bf 1}) \vee y, \  x \wedge (y + \mu {\bf 1}) \in S ,
\label{trsubmset}
\\
x,y \in S, \  \suppp(x-y)\not= \emptyset 
& \ \Longrightarrow \
x - \unitvec{A}, \  y + \unitvec{A} \  \in S
\mbox{ for } A = \argmax_{i} \{ x_{i} - y_{i} \},
\label{lnatsetAPR}
\end{align}
which correspond to submodularity \eqref{submfn},
 translation-submodularity \eqref{trsubmfn}, and the condition \eqref{lnatfnAPR},
respectively.
The first condition \eqref{submset} means that $S$ forms a sublattice of $\ZZ\sp{n}$.

\begin{proposition}  \label{PRlnatset}
For a nonempty set $S \subseteq  \ZZ\sp{n}$ 
the following conditions, {\rm (a)} to {\rm (d)}, are equivalent:


{\rm (a)}
$S$ is an \Lnat-convex set, that is,
it satisfies {\rm (\ref{midptcnvset})}.

{\rm (b)}
$S$ is an integrally convex set that
satisfies \eqref{submset}.

{\rm (c)}
$S$ satisfies \eqref{trsubmset} for all nonnegative $\mu \in \ZZ$.

{\rm (d)}
$S$ satisfies \eqref{lnatsetAPR}.
\finbox
\end{proposition}

The concept of \Lnat-convex functions
was introduced in \cite{FM00}
as an equivalent variant of 
{\rm L}-convex functions introduced earlier in \cite{Mdca98}.
L- and \Lnat-convex functions form  major classes of discrete convex functions
\cite[Chapter 7]{Mdcasiam}.
They have applications in several fields including
image processing,
auction theory, 
inventory theory, 
and scheduling
\cite{Mdcaeco16,Shi17L,SCB14}.

\subsubsection{L-convex sets and functions}

A function 
$f(x_{1},x_{2}, \ldots, x_{n} )$
with $\dom f \not= \emptyset$ is called {\em L-convex}
if it is submodular \eqref{submfn}
 and there exists
$r \in \RR$ such that 
\begin{equation}\label{shiftlfnZ}
f(x + \vecone) = f(x) +  r
\end{equation}
for all $x  \in \ZZ\sp{n}$.
If $f$ is {\rm L}-convex,
the function
$g(x_{2}, \ldots, x_{n} ) := f(0, x_{2}, \ldots, x_{n} )$
is an \Lnat-convex function, and any \Lnat-convex function 
arises in this way.
The function $\tilde f$ in (\ref{lfnlnatfnrelation}) derived from an \Lnat-convex function $f$  
is an {\rm L}-convex function
with $\tilde f(x_{0}+1, x + \vecone) = \tilde f(x_{0},x)$,
and we have $f(x)  =  \tilde f(0,x)$.

\begin{theorem} \label{THlfncond}
For $f: \ZZ^{n} \to \RR \cup \{ +\infty \}$
the following conditions, {\rm (a)} to {\rm (c)}, are equivalent:


{\rm (a)}
$f$ is an L-convex function, that is, it satisfies \eqref{submfn} 
and \eqref{shiftlfnZ} for some $r \in \RR$.

{\rm (b)}
$f$ is an \Lnat-convex function that satisfies \eqref{shiftlfnZ} for some $r \in \RR$.

{\rm (c)}
$f$ satisfies translation-submodularity \eqref{trsubmfn}
for all $\mu \in \ZZ$ 
(including $\mu < 0$).
\finbox
\end{theorem}

A nonempty set $S$ is called {\em L-convex} if its indicator function
$\delta_{S}$
is an L-convex function.
This means that $S$ is L-convex if and only if 
it satisfies 
\eqref{submset} and
\begin{equation} \label{invarone}
 x  \in S
\ \Longrightarrow \
 x -  {\bf 1},  \  x +  {\bf 1} \in S .
\end{equation}
The effective domain of an L-convex function is an L-convex set.

The following proposition gives equivalent conditions for L-convex sets.

\begin{proposition}  \label{PRlset}
For a nonempty set $S \subseteq  \ZZ\sp{n}$ 
the following conditions, {\rm (a)} to {\rm (c)}, are equivalent:


{\rm (a)}
$S$ is an L-convex set, that is, it satisfies \eqref{submset} and \eqref{invarone}.

{\rm (b)}
$S$ is an \Lnat-convex set that satisfies \eqref{invarone}.

{\rm (c)}
$S$ satisfies \eqref{trsubmset} for all $\mu \in \ZZ$
(including $\mu < 0$).
\finbox
\end{proposition}

\subsubsection{Discrete midpoint convex sets and functions}

A generalization of the concept of \Lnat-convexity has been 
introduced recently in \cite{MMTT19dmpc}.
A nonempty set $S \subseteq  \ZZ\sp{n}$ is said to be {\em discrete midpoint convex} if
\begin{equation} \label{dirintcnvsetdef}
 x, y \in S, \ \| x - y \|_{\infty} \geq 2
\ \Longrightarrow \
\left\lceil \frac{x+y}{2} \right\rceil ,
\left\lfloor \frac{x+y}{2} \right\rfloor  \in S.
\end{equation}
This condition is  weaker than the defining condition (\ref{midptcnvset})
for an \Lnat-convex set,
and hence every \Lnat-convex set
is a discrete midpoint convex set.

A function $f: \ZZ\sp{n} \to \RR \cup \{ +\infty \}$
is called {\em  globally discrete midpoint convex} if
the discrete midpoint convexity 
\begin{equation} \label{midptcnvfn2}
 f(x) + f(y) \geq
   f \left(\left\lceil \frac{x+y}{2} \right\rceil\right) 
  + f \left(\left\lfloor \frac{x+y}{2} \right\rfloor\right) 
\end{equation}
is satisfied by every pair $(x, y) \in \ZZ\sp{n} \times \ZZ\sp{n}$
with $\| x - y \|_{\infty} \geq 2$.
The effective domain of a
globally discrete midpoint convex function
is necessarily a discrete midpoint convex set.
A function $f: \ZZ\sp{n} \to \RR \cup \{ +\infty \}$
is called {\em  locally discrete midpoint convex}
if $\dom f$ is a discrete midpoint convex set
and the discrete midpoint convexity (\ref{midptcnvfn2})
is satisfied by every pair $(x, y) \in \ZZ\sp{n} \times \ZZ\sp{n}$
with $\| x - y \|_{\infty} = 2$ (exactly equal to $2$).
Obviously, 
every \Lnat-convex function is
globally discrete midpoint convex, and 
every globally discrete midpoint convex function is
locally discrete midpoint convex.

We sometimes abbreviate ``discrete midpoint convex(ity)'' to ``d.m.c.''

\subsection{M-convexity}

\subsubsection{\Mnat-convex sets and functions}
\label{SCmnatfn}

We say that a function
$f: \ZZ\sp{N} \to \RR \cup \{ +\infty \}$
with $\dom f \not= \emptyset$
is {\em M$\sp{\natural}$-convex}, if,
for any $x, y \in \ZZ\sp{N}$ and $i \in \suppp(x-y)$, 
we have (i)
\begin{equation}  \label{mconvex1Z}
f(x) + f(y)  \geq  f(x -\unitvec{i}) + f(y+\unitvec{i})
\end{equation}
or (ii) there exists some $j \in \suppm(x-y)$ such that
\begin{equation}  \label{mconvex2Z}
f(x) + f(y)   \geq 
 f(x-\unitvec{i}+\unitvec{j}) + f(y+\unitvec{i}-\unitvec{j}) .
\end{equation}
This property is referred to as the {\em exchange property}.
See Fig.~\ref{FGmnatfndef}, in which 
$(x',y')= (x -\unitvec{i}, y+\unitvec{i})$ 
for \eqref{mconvex1Z} and
$(x'',y'')= ( x-\unitvec{i}+\unitvec{j}, y+\unitvec{i}-\unitvec{j})$
for \eqref{mconvex2Z}.

\begin{figure}\begin{center}
\includegraphics[width=0.35\textwidth,clip]{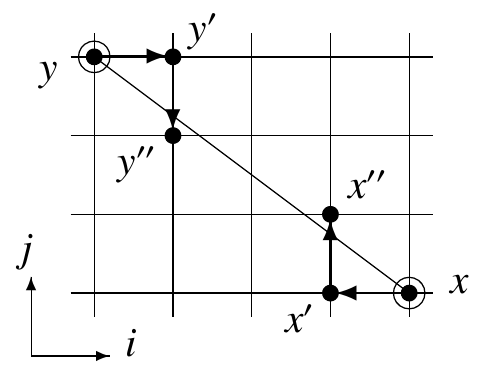}
\\
$(x',y')= (x -\unitvec{i}, y+\unitvec{i})$, \ 
$(x'',y'')= ( x-\unitvec{i}+\unitvec{j}, y+\unitvec{i}-\unitvec{j})$
\caption{Definition of M$\sp{\natural}$-convex functions}
 \label{FGmnatfndef}
\end{center}\end{figure}

A more compact expression of the exchange property is as follows:
\begin{description}
\item[\Mnvexb]
 For any $x, y \in \ZZ\sp{N}$ and $i \in \suppp(x-y)$, we have
\begin{equation} \label{mnconvexc2Z}
f(x) + f(y)   \geq 
\min_{j \in \suppm(x - y) \cup \{ 0 \}} 
 \{ f(x - \unitvec{i} + \unitvec{j}) + f(y + \unitvec{i} - \unitvec{j}) \},
\end{equation}
\end{description}
where $\unitvec{0}=\veczero$ (zero vector).
In the above statement we may change
``For any $x, y \in \ZZ\sp{N}$\,'' to ``For any $x, y \in \dom f$\,''
 since if $x \not\in \dom f$ or $y \not\in \dom f$,
the inequality (\ref{mnconvexc2Z}) trivially holds with $f(x) + f(y)  = +\infty$.

It follows from \Mnvex that the effective domain $S = \dom f$
of an M$\sp{\natural}$-convex function $f$
has the following exchange property:
\begin{description}
\item[\Bnvexb] 
For any $x, y \in S$ and $i \in \suppp(x-y)$, we have
(i)
$x -\unitvec{i} \in S$ and $y+\unitvec{i} \in S$
\ or \  
\\
(ii) there exists some $j \in \suppm(x-y)$ such that
$x-\unitvec{i}+\unitvec{j}  \in S$ and $y+\unitvec{i}-\unitvec{j} \in S$.
\end{description}
A nonempty set $S \subseteq \ZZ\sp{N}$ having this property 
is called an {\em M$\sp{\natural}$-convex set},
which is an alias for the set of integer points in an integral generalized polymatroid.

M$\sp{\natural}$-convex functions can be characterized by 
a number of different exchange properties including a local exchange property
under the assumption that 
function $f$ is (effectively) defined on an M$\sp{\natural}$-convex set.
See \cite{MS18mnataxiom} 
as well as \cite[Theorem~4.2]{Mdcaeco16} and \cite[Theorem~6.8]{ST15jorsj}.

It is known \cite[Theorem 3.8]{MS01rel} 
(see also \cite[Theorem 6.19]{Mdcasiam})
that an M$\sp{\natural}$-convex function
$f: \ZZ\sp{N} \to \RR \cup \{ +\infty \}$
 is {\em supermodular} on the integer lattice, i.e.,
\begin{equation} \label{setfnsubmZ}
 f(x) + f(y) \leq f(x \vee y) + f(x \wedge y)
 \qquad (x, y \in \ZZ\sp{N}) .
\end{equation}
Not every supermodular function is M$\sp{\natural}$-convex,
that is, M$\sp{\natural}$-convex functions
form a proper subclass of supermodular functions on $\ZZ\sp{N}$.

The concept of \Mnat-convex functions
was introduced in \cite{MS99gp}
as an equivalent variant of 
{\rm M}-convex functions introduced earlier in \cite{Mstein96}.
{\rm M}- and \Mnat-convex functions form major classes of discrete convex functions
\cite[Chapter 6]{Mdcasiam}.
M-convexity and its various variants have applications in 
economics
\cite{Mdcasiam,Mbonn09,Mdcaeco16,MTcompeq03,ST15jorsj}.
See 
\cite[Chapters 4 and 6]{Mdcasiam},
\cite[Section 4]{Mdcaeco16},
\cite{MS18mnataxiom}
for detailed discussion about the exchange properties.

\subsubsection{M-convex sets and functions}

If a set $S \subseteq \ZZ\sp{N}$ lies on a hyperplane with a constant component sum
(i.e., $x(N) = y(N)$ for all  $x, y \in S$),
the exchange property \Bnvex
takes a simpler form
(without the possibility of the first case (i)): 
\begin{description}
\item[\Bvexb] 
For any $x, y \in S$ and $i \in \suppp(x-y)$, 
there exists some $j \in \suppm(x-y)$ such that
$x-\unitvec{i}+\unitvec{j}  \in S$ and $y+\unitvec{i}-\unitvec{j} \in S$.
\end{description}
A nonempty set $S \subseteq \ZZ\sp{N}$ having this exchange property 
is called an {\em M-convex set},
which is an alias for the set of integer points in an integral base polyhedron.  

An M$\sp{\natural}$-convex function 
whose effective domain is an M-convex set
is called an {\em M-convex function}
\cite{Mstein96,Mdca98,Mdcasiam}.
In other words,  a function 
$f: \ZZ\sp{N} \to \RR \cup \{ +\infty \}$
is M-convex
if and only if it satisfies
the exchange property:
\begin{description}
\item[\Mvexb]
 For any $x, y \in \dom f$  and $i \in \suppp(x-y)$, 
there exists
$j \in \suppm(x-y)$ such that
\eqref{mconvex2Z} holds.
\end{description}
M-convex functions can be characterized by a local exchange property
under the assumption 
that function $f$ is (effectively) defined on an M-convex set.
See \cite[Section 6.2]{Mdcasiam}.

When the effective domain of a function $f$ is contained in the unit cube $\{ 0, 1 \}\sp{N}$,
$f$ is M-convex 
if and only if
$-f$ is a valuated matroid introduced earlier in \cite{DW90,DW92}.
Historically, the concept of M-concave functions was introduced in \cite{Mstein96} 
as a generalization of valuated matroids.
See also \cite[Chapter 5]{Mspr2000} for valuated matroids. 

M-convex functions and
M$\sp{\natural}$-convex functions are equivalent concepts,
in that M$\sp{\natural}$-convex functions in $n$ variables 
can be obtained as projections of M-convex functions in $n+1$ variables.
More formally, let ``$0$'' denote a new element not in $N$ and
$\tilde{N} = \{0\} \cup N$.
A function $f : \ZZ\sp{N} \to \RR \cup \{ +\infty \}$ is M$\sp{\natural}$-convex
if and only if  the function 
$\tilde{f} : \ZZ\sp{\tilde{N}} \to \RR \cup \{ +\infty \}$ 
defined by
\begin{equation} \label{mfnmnatfnrelationvex}
  \tilde{f}(x_{0},x) = \left\{ \begin{array}{ll}
      f(x)    & \mbox{ if $x_{0} = {-}x(N)$} \\
      +\infty & \mbox{ otherwise}
    \end{array}\right.
 \qquad ( x_{0} \in \ZZ, x \in \ZZ\sp{N})
\end{equation}
is an M-convex function.

\subsubsection{M-convex functions on jump systems}

Let $x$ and $y$ be integer vectors.
The smallest integer box containing $x$ and $y$
is given by $[x \wedge y,x \vee y]_{\ZZ}$
using the notations $x \wedge y$ and $x \vee y$ introduced in \eqref{veewedgedef}.
A vector $s \in \ZZ\sp{N}$ is called an {\em $(x,y)$-increment} if
$s= \unitvec{i}$ or $s= -\unitvec{i}$ for some $i \in N$ and 
$x+s \in [x \wedge y,x \vee y]_{\ZZ}$.

A nonempty set $S \subseteq \ZZ\sp{N}$ is said 
to be a {\em jump system} 
\cite{BouC95}
if satisfies an exchange axiom,
called the {\em 2-step axiom}:
\begin{description}
\item[(2-step axiom)]
For any  $x, y \in S$
and any $(x,y)$-increment $s$ with $x+s \not\in S$,
there exists an $(x+s,y)$-increment $t$ such that $x+s+t \in S$.
\end{description}
Note that we have the possibility of $s = t$ in the 2-step axiom.

A set $S \subseteq \ZZ\sp{N}$ is called a {\em constant-sum system}
if $x(N)=y(N)$ for any $x,y \in S$.
A constant-sum jump system is nothing but an M-convex set,
since an M-convex set is a constant-sum system
and, for a constant-sum system, the 2-step axiom reduces to 
the exchange axiom \Bvex for an M-convex set.

A set $S \subseteq \ZZ\sp{N}$ is called a {\em constant-parity system}
if $x(N)-y(N)$ is even for any $x,y \in S$.
For a constant-parity system $S$,
the 2-step axiom is simplified to:
\begin{description}
\item[(J-EXC$_{+}$)]
For any  $x, y \in S$
and for any $(x,y)$-increment $s$,
there exists an $(x+s,y)$-increment $t$ such that $x+s+t \in S$.
\end{description}
It is known (\cite[Lemma~2.1]{Mmjump06})
that this exchange property (J-EXC$_{+}$) is equivalent to 
the following (seemingly stronger) exchange property:
\begin{description}
\item[(J-EXC)]
For any  $x, y \in S$ and any $(x,y)$-increment $s$,
there exists an $(x+s,y)$-increment $t$ such that
$x+s+t \in S$ and $y-s-t \in S$.
\end{description}
That is, (J-EXC) characterizes 
a {\em constant-parity jump system}
(or {\em c.p.~jump system} for short).

\begin{example}  \rm  \label{EXjumpdim1}
Let  $S = \{ 0, 2 \}$, which is a subset  of $\ZZ$ (with $n=1$).
This is a constant-parity jump system.
Indeed, 
for $(x,y,s)=(0,2,1)$ we can take $t=1$ in (J-EXC),
and for $(x,y,s)=(2,0,-1)$ we can take $t=-1$ in (J-EXC).
In contrast,
this set is not an \Mnat-convex set,
since we cannot take $t=s$ in 
the exchange property 
\Bnvex for \Mnat-convex sets.
\finbox
\end{example}

A function  $f: \ZZ\sp{N} \to \RR \cup \{ +\infty \}$
is called%
\footnote{
This concept (``jump M-convex function'')  is the same as ``M-convex function on a jump system"
in \cite{KMT07jump,Mmjump06}.
} 
{\em jump M-convex} 
if it satisfies
the following exchange axiom:
\begin{description}
\item[(JM-EXC)]
For any  $x, y \in \dom f$
and any $(x,y)$-increment $s$,
 there exists an $(x+s,y)$-increment $t$ such that
$x+s+t \in \dom f$,  \  $y-s-t \in \dom f$, and
\begin{equation}  \label{jumpmexc2}
 f(x)+f(y) \geq  f(x+s+t)  + f(y-s-t) .
\end{equation} 
\end{description}
The effective domain of a jump M-convex function is 
a constant-parity jump system.

Just as we consider \Mnat-convex functions as well as
M-convex functions,
we can introduce 
the concept of {\em jump \Mnat-convex functions}
by the following exchange axiom
\cite{Mmnatjump19}:
\begin{description}
\item[(J\Mnat-EXC)]
For any  $x, y \in \dom f$
and any $(x,y)$-increment $s$, we have
\\
(i) $x+s \in \dom f$, \  $y-s \in \dom f$, and 
\begin{equation}  \label{jumpmexc1}
 f(x)+f(y) \geq  f(x+s)  + f(y-s) , 
\end{equation}
or (ii) there exists an $(x+s,y)$-increment $t$ such that
$x+s+t \in \dom f$,  \  $y-s-t \in \dom f$, and
\eqref{jumpmexc2} holds.
\end{description}

The condition 
(J\Mnat-EXC) is weaker than (JM-EXC), and hence
every jump M-convex function is a jump \Mnat-convex function.
However, the concepts of jump M-convexity and
jump \Mnat-convexity are in fact equivalent to each other in the sense that 
jump \Mnat-convex function in $n$ variables 
can be identified with jump M-convex functions in $n+1$ variables.
More specifically, for any integer vector $x \in \ZZ\sp{N}$
we define $\pi(x)=0$ if $x(N)$ is even, and
$\pi(x)=1$ if $x(N)$ is odd,
and let
$\tilde{N} = \{0\} \cup N$
with a new element ``$0$'' not in $N$.
It is known \cite{Mmnatjump19} that 
a function $f : \ZZ\sp{N} \to \RR \cup \{ +\infty \}$ is 
jump \Mnat-convex
if and only if  the function 
$\tilde{f} : \ZZ\sp{\tilde{N}} \to \RR \cup \{ +\infty \}$ 
defined by
\begin{equation}  \label{ftildeJ}
\tilde f(x_{0}, x)= 
   \left\{  \begin{array}{ll}
   f(x)            &   (x_{0} = \pi(x))       \\
   + \infty      &   (\mbox{\rm otherwise})  \\
                      \end{array}  \right.
 \qquad ( x_{0} \in \ZZ, x \in \ZZ\sp{N})
\end{equation}
is a jump M-convex function.

As a consequence of (J\Mnat-EXC),
the effective domain of a jump \Mnat-convex function
is a jump system that satisfies 
\begin{description}
\item[(J$\sp{\natural}$-EXC)]
For any  $x, y \in S$ and any $(x,y)$-increment $s$, we have
(i) $x+s \in S$ and $y-s \in S$, or 
(ii)
there exists an $(x+s,y)$-increment $t$ such that
$x+s+t \in S$ and $y-s-t \in S$.
\end{description}
A jump system that satisfies
(J$\sp{\natural}$-EXC) 
is called a {\em simultaneous exchange jump system}
(or {\em s.e.~jump system} for short) \cite{Mmnatjump19}.
Every constant-parity jump system
is a simultaneous exchange jump system,
since the condition (J-EXC) implies (J$\sp{\natural}$-EXC).

When a set $S$ is a subset of $\{ 0, 1 \}\sp{N}$,
$S$ is a c.p.~jump system
if and only if
it is an even delta-matroid \cite{BouC95},
and $S$ is an s.e.~jump system
if and only if
it is a simultaneous delta-matroid considered in \cite{Tak14fores},
where a subset of $N$ is identified with its characteristic vector.
Furthermore, 
when the effective domain of a function $f$ is contained in the unit cube $\{ 0, 1 \}\sp{N}$,
$f$ is jump M-convex 
if and only if
$-f$ is a valuated delta-matroid of \cite{DW91valdel},
and $f$ is jump \Mnat-convex 
if and only if
$-f$ is a valuation on a simultaneous delta-matroid in the sense of \cite{Tak14fores}.

Not every jump system is a simultaneous exchange jump system,
as the following examples show.

\begin{example}  \rm  \label{EXnonsejumpdim1}
Let  $S = \{ 0, 2, 3 \}$, which is a subset of $\ZZ$ (with $n=1$).
This set satisfies the 2-step axiom, and hence a jump system.
However, it does not satisfy
the simultaneous exchange property (J$\sp{\natural}$-EXC).
Indeed, (J$\sp{\natural}$-EXC) fails for $x=0$, $y=3$, and $s=1$.
\finbox
\end{example}

\begin{example}[{\cite{Tak14fores}}]  \rm  \label{EXnonsejump}
Let 
$S = \{ (0,0,0), (1,1,0), (1,0,1), (0,1,1), (1,1,1) \}$.
This set satisfies the 2-step axiom, and hence a jump system.
However, it does not satisfy
the simultaneous exchange property (J$\sp{\natural}$-EXC).
Indeed, (J$\sp{\natural}$-EXC) fails 
for $x=(0,0,0)$, $y=(1,1,1)$, and $s=(1,0,0)$.
It is worth noting that 
$S$ consists of the characteristic vectors of 
the rows (and columns) of nonsingular principal minors of 
the symmetric matrix
\[ 
A = \left[ \begin{array}{ccc}
0 & 1 & 1 \\
1 & 0 & 1 \\
1 & 1 & 0 \\
\end{array} \right]
\]
and hence it is a delta-matroid.
\finbox
\end{example}

The following example demonstrates the difference of
\Mnvex and (J\Mnat-EXC) for functions.

\begin{example}  \rm  \label{EXjmnatmnat}
Let 
$S=\{ (0,0), (1,0), (0,1), (1,1) \}$ 
and define $f: \ZZ\sp{2} \to \RR \cup \{ +\infty \}$ 
by
$f(0,0)=f(1,1)=a$ and $f(1,0)=f(0,1)=b$
with  $\dom f = S$.
\Mnvex is satisfied
if and only if $a \geq b$, whereas
(J\Mnat-EXC) is true for any $(a, b)$.
\finbox
\end{example}

The inclusion relations for sets and functions may be summarized as follows:
\begin{align}
& 
\{ \mbox{\rm M-convex sets} \}  \subsetneqq \
  \left\{  \begin{array}{l}  \{ \mbox{\rm \Mnat-convex  sets} \}
                          \\ \{ \mbox{\rm c.p. jump systems} \} \end{array}  \right \}
\subsetneqq \   \{ \mbox{\rm s.e. jump systems} \} 
\subsetneqq \   \{ \mbox{\rm jump systems} \} ,
\label{setfamjump0}
\\
&
\{ \mbox{\rm M-convex fns} \}  \subsetneqq \
  \left\{  \begin{array}{l}  \{ \mbox{\rm \Mnat-convex fns} \}
                          \\ \{ \mbox{\rm jump M-convex fns} \} \end{array}  \right \}
\subsetneqq \  
\{ \mbox{\rm jump \Mnat-convex fns} \} .
\label{fnfamjump0}
\end{align}
It is noted that no convexity class is introduced for functions defined on general jump systems.

Jump M- and \Mnat-convex functions find applications in several fields including
matching theory \cite{BK12,KST12cunconj,KT09evenf,Tak14fores}
and algebra \cite{Bra10halfplane}.

\subsection{Multimodularity}

Recall that $\unitvec{i}$ denotes the $i$th unit vector for $i =1,2,\ldots, n$,
and 
$\mathcal{F} \subseteq \ZZ\sp{N}$ 
be the set of vectors defined by
\begin{equation} \label{multimodirection1}
\mathcal{F} = \{ -\unitvec{1}, \unitvec{1}-\unitvec{2}, \unitvec{2}-\unitvec{3}, \ldots, 
  \unitvec{n-1}-\unitvec{n}, \unitvec{n} \} .
\end{equation}
A finite-valued function $f: \ZZ\sp{n} \to \RR$
is said to be {\em multimodular}
\cite{Haj85}
if it satisfies
\begin{equation} \label{multimodulardef1}
 f(z+d) + f(z+d') \geq   f(z) + f(z+d+d')
\end{equation}
for all 
$z \in \ZZ\sp{n}$ and all distinct $d, d' \in \mathcal{F}$.
It is known \cite[Proposition 2.2]{Haj85} 
that $f: \ZZ\sp{n} \to \RR$
is multimodular if and only if the function 
$\tilde f: \ZZ\sp{n+1} \to \RR$ 
defined by 
\begin{equation} \label{multimodular1}
 \tilde f(x_{0}, x) = f(x_{1}-x_{0},  x_{2}-x_{1}, \ldots, x_{n}-x_{n-1})  
 \qquad ( x_{0} \in \ZZ, x \in \ZZ\sp{n})
\end{equation}
is submodular in $n+1$ variables. 
This characterization enables us to define multimodularity 
for a function that may take the infinite value $+\infty$.
That is, we say
\cite{MM19multm,Mmult05} 
that a function $f: \ZZ\sp{n} \to \RR \cup \{ +\infty \}$ with $\dom f \not= \emptyset$
is multimodular if the function 
$\tilde f: \ZZ\sp{n+1} \to \RR \cup \{ +\infty \}$
associated with $f$ by (\ref{multimodular1}) is submodular.

Multimodularity and L$\sp{\natural}$-convexity
have the following close relationship.

\begin{theorem}[\cite{Mmult05}]  \label{THmmfnlnatfn}
A function $f: \ZZ\sp{n} \to \RR \cup \{ +\infty \}$
is multimodular if and only if the function $g: \ZZ\sp{n} \to \RR \cup \{ +\infty \}$
defined by
\begin{equation} \label{mmfnGbyF}
 g(p) = f(p_{1}, \  p_{2}-p_{1}, \  p_{3}-p_{2}, \ldots, p_{n}-p_{n-1})  
 \qquad ( p \in \ZZ\sp{n})
\end{equation}
is L$\sp{\natural}$-convex.
\finbox
\end{theorem}

Note that the relation (\ref{mmfnGbyF}) between $f$ and $g$ can be rewritten as
\begin{equation} \label{mmfnFbyG}
 f(x) = g(x_{1}, \  x_{1}+x_{2}, \  x_{1}+x_{2}+x_{3}, 
   \ldots, x_{1}+ \cdots + x_{n})  
 \qquad ( x \in \ZZ\sp{n}) .
\end{equation}
Using a bidiagonal matrix 
$D=(d_{ij} \mid 1 \leq i,j \leq n)$ defined by
\begin{equation} \label{matDdef}
 d_{ii}=1 \quad (i=1,2,\ldots,n),
\qquad
 d_{i+1,i}=-1 \quad (i=1,2,\ldots,n-1),
\end{equation}
we can express (\ref{mmfnGbyF}) and (\ref{mmfnFbyG}) 
more compactly  as $g(p)=f(Dp)$ and $f(x)=g(D\sp{-1}x)$, respectively. 
The matrix $D$ is unimodular, and its inverse
$D\sp{-1}$ is an integer matrix
with $(D\sp{-1})_{ij}=1$ for $i \geq j$ and 
$(D\sp{-1})_{ij}=0$ for $i < j$.
For $n=4$, for example, we have
\[
D = {\small
\left[ \begin{array}{rrrr}
1 & 0 & 0 & 0 \\
-1 & 1 & 0 & 0 \\
0 & -1 & 1 & 0 \\
0 & 0 & -1 & 1 \\
\end{array}\right]},
\qquad
D\sp{-1} = {\small
\left[ \begin{array}{rrrr}
1 & 0 & 0 & 0 \\
1 & 1 & 0 & 0 \\
1 & 1 & 1 & 0 \\
1 & 1 & 1 & 1 \\
\end{array}\right]}.
\]

A nonempty set $S$ is called {\em multimodular}
if its indicator function $\delta_{S}$ is multimodular. 
A multimodular set $S$ can be represented as 
$S = \{ D p \mid p \in T \}$
for some \Lnat-convex set $T$,
where $T$ is uniquely determined from $S$ as 
$T = \{ D\sp{-1} x \mid x \in S \}$.
It follows from (\ref{mmfnGbyF}) that
the effective domain of a multimodular function is a multimodular set.

A quadratic multimodular function admits a simple characterization 
in terms of its coefficient matrix.

\begin{proposition}[{\cite[Proposition 3]{MM19multm}}] \rm \label{PRmmfnquadr}
A quadratic function $f(x) = x^{\top} A x$ is multimodular
if and only if
\begin{equation} \label{mmfquadrcond}
a_{ij} - a_{i,j+1} - a_{i+1,j} + a_{i+1,j+1} \leq 0
\qquad (0 \leq i < j \leq n),
\end{equation}
where $A=(a_{ij} \mid i,j =1,2,\ldots, n)$ and $a_{ij} =0$ if $i=0$ or $j=n+1$.
\finbox
\end{proposition}

The reader is referred to 
\cite{AGH00,AGH03,Haj85,Mmult05,Mdcaprimer07}
for more about multimodularity,
and \cite{MM19multm} 
for a recent development.

\subsection{Relations among classes of discrete convex sets}
\label{SCclasssetZ}

Figure \ref{FGdcsetclassZR} shows the inclusion relations 
among the most fundamental classes of discrete convex sets.
Integrally convex sets
contain both \Lnat-convex sets and \Mnat-convex sets.
\Lnat-convex sets contain L-convex sets as a special case.
Similarly,
\Mnat-convex sets contain M-convex sets as a special case.
The classes of L-convex sets and M-convex sets
are disjoint, whereas the intersection of the classes 
of \Lnat-convex sets and \Mnat-convex sets
is precisely the class of integer boxes.
Integer boxes are neither L-convex nor M-convex,
with the exception that the entire lattice
$\ZZ\sp{n}$ is an L-convex set and every singleton set is an M-convex set.

\begin{figure}\begin{center}
\includegraphics[width=0.7\textwidth,clip]{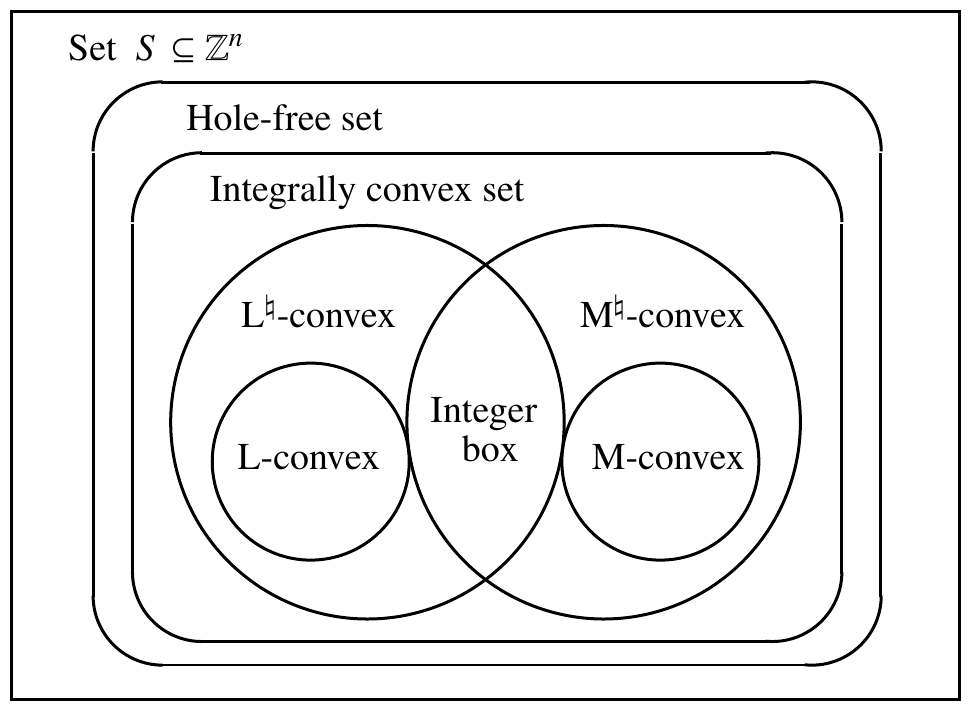}
\caption{Classes of discrete convex sets
(\Lnat-convex $\cap$ \Mnat-convex $=$ integer box)}
 \label{FGdcsetclassZR}
\end{center}\end{figure}

Other kinds of discrete convex sets treated in this paper
are multimodular sets,
discrete midpoint convex sets, and jump systems.
Multimodular sets can be obtained from \Lnat-convex sets by a unimodular transformation,
and vice versa.
Integer boxes are multimodular, and multimodular sets are integrally convex.
\Lnat-convex sets are discrete midpoint convex,
and discrete midpoint convex sets are integrally convex.
Simultaneous exchange jump systems contain \Mnat-convex sets as a special case,
and
constant-parity jump systems contain M-convex sets as a special case.
Jump systems,
whether simultaneous exchange or constant-parity,  
are not necessarily integrally convex.

The inclusion relations among these set classes
are summarized in the following theorem.
We mention that \eqref{setfamL} to \eqref{setfamsepM}
are shown in Fig.~\ref{FGdcsetclassZR}.

\begin{theorem} \label{THsetclassinclusion}
The following inclusion relations hold for subsets of $\ZZ\sp{n}$:
\begin{align}
&
\{ \mbox{\rm {\rm L}-convex} \} \subsetneqq \ 
\{ \mbox{\rm \Lnat-convex} \} \subsetneqq \ \{ \mbox{\rm integrally convex} \},
\label{setfamL}
\\ &
\{ \mbox{\rm {\rm M}-convex} \} \subsetneqq \ 
\{ \mbox{\rm \Mnat-convex} \}
\subsetneqq \ \{ \mbox{\rm integrally convex} \},
\label{setfamM}
\\ &
\{ \mbox{\rm integer box} \} = \{ \mbox{\rm \Lnat-convex} \} \cap  \{ \mbox{\rm \Mnat-convex} \},
\label{setfamsepLnMn}
\\ &
\{ \mbox{\rm integer box} \} \cap  \{ \mbox{\rm L-convex} \} = \{ \ZZ\sp{n} \},
\label{setfamsepL}
\\ &
\{ \mbox{\rm integer box} \} \cap \{ \mbox{\rm M-convex} \}
 = \{ \mbox{\rm singleton} \},
\label{setfamsepM}
\\ &
\{ \mbox{\rm integer box} \}
\subsetneqq \ \{ \mbox{\rm multimodular} \}
\subsetneqq \ \{ \mbox{\rm integrally convex} \} ,
\label{setfammultm}
\\ &
\{ \mbox{\rm \Lnat-convex} \}
\subsetneqq \ \{ \mbox{\rm discrete midpoint convex} \}
\subsetneqq \ \{ \mbox{\rm integrally convex} \},
\label{setfamdmc}
\\ &
\{ \mbox{\rm M-convex} \}  \subsetneqq \
  \left\{  \begin{array}{l}  \{ \mbox{\rm \Mnat-convex} \}
                          \\ \{ \mbox{\rm c.p. jump} \} \end{array}  \right \}
\subsetneqq \  
\{ \mbox{\rm s.e. jump} \} 
\not\subseteq \ \{ \mbox{\rm integrally convex} \},
\label{setfamjump}
\end{align}
where ``c.p.~jump" and ``s.e.~jump" 
mean constant-parity jump system and simultaneous exchange jump system, 
respectively. 
\finbox
\end{theorem}

\begin{remark} \rm \label{RMsetclassbib}
Here is a supplement to Theorem~\ref{THsetclassinclusion}.
The relation \eqref{setfamsepLnMn} originates in \cite[Lemma 5.7]{MS01rel}.
The inclusion relations given in 
Theorem~\ref{THsetclassinclusion} follow from the inclusion relations 
for discrete convex functions in Theorem~\ref{THfnclassinclusion}.
Integral convexity of L-convex sets 
is proved in \cite[Theorem 5.10]{Mdcasiam},
and that of M-convex sets in \cite[Theorem 4.24]{Mdcasiam}.
Integral convexity of discrete midpoint convex sets 
is established in
\cite[Proposition 1]{MMTT19dmpc}.
\finbox
\end{remark}

\begin{remark} \rm \label{RMsetdim2}
For subsets of $\ZZ\sp{2}$,
\Lnat-convexity and \Mnat-convexity are essentially the same 
in the sense that 
$S \subseteq \ZZ\sp{2}$ is \Lnat-convex if and only if
$T = \{ (x_{1}, -x_{2}) \mid  (x_{1}, x_{2}) \in S \}$
is \Mnat-convex.
Moreover, multimodularity is the same as \Mnat-convexity,
that is, a set $S \subseteq \ZZ\sp{2}$
is multimodular if and only if it is \Mnat-convex.
\finbox
\end{remark}

\paragraph*{Hole-free property}

Integrally convex sets are hole-free
($S =  \overline{S} \cap \ZZ^{n}$),
and hence the same is true for 
L-convex, \Lnat-convex, M-convex, \Mnat-convex, 
multimodular, and discrete midpoint convex sets.
In contrast, jump systems may have a hole in its convex hull.

\subsection{Relations among classes of discrete convex functions}
\label{SCclassfnZ}

Figure \ref{FGdcfclassZR} shows the inclusion relations 
among the most fundamental classes of discrete convex functions.
Integrally convex functions
contain both \Lnat-convex functions and \Mnat-convex functions.
\Lnat-convex functions contain L-convex functions as a special case.
Similarly,
\Mnat-convex functions contain M-convex functions as a special case.
The classes of L-convex functions and M-convex functions
are disjoint, whereas the intersection of the classes 
of \Lnat-convex functions and \Mnat-convex functions
is precisely the class of separable convex functions.
A separable convex function is neither L-convex nor M-convex,
except that every linear function on the entire lattice $\ZZ\sp{n}$
is an L-convex function and every function on a singleton set is an M-convex function. 

\begin{figure}\begin{center}
\includegraphics[width=0.7\textwidth,clip]{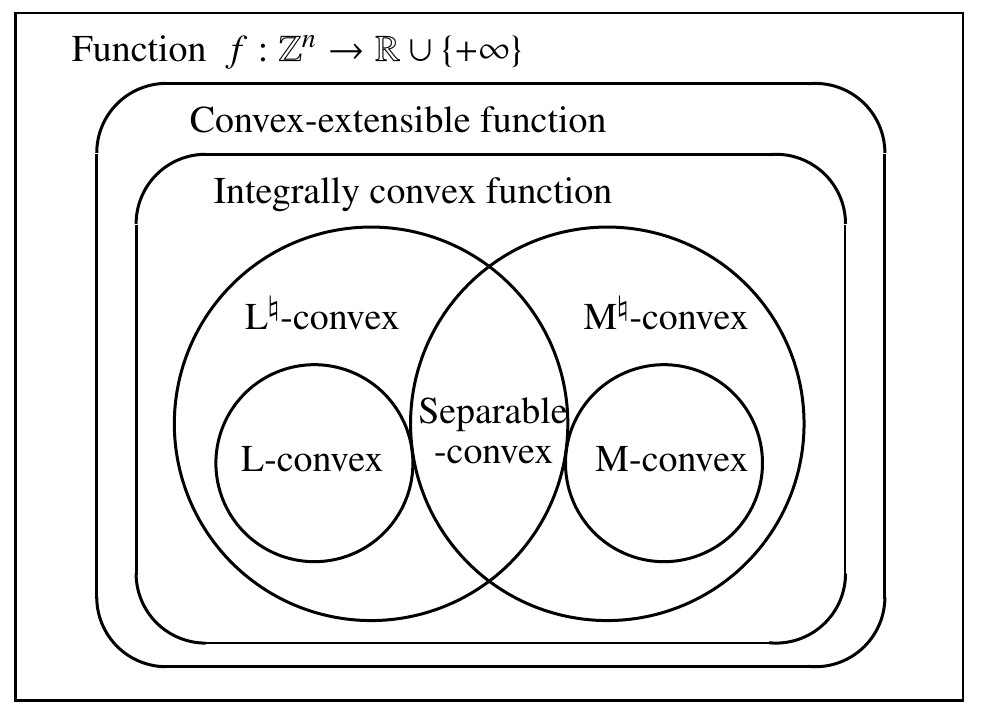}
\caption{Classes of discrete convex functions
(\Lnat-convex $\cap$ \Mnat-convex $=$ separable convex)}
 \label{FGdcfclassZR}
\end{center}\end{figure}

Other kinds of discrete convex functions treated in this paper
are multimodular functions,
globally and locally discrete midpoint convex functions, and
jump \Mnat- and M-convex functions.
Multimodular functions can be obtained from \Lnat-convex functions  
by a unimodular coordinate transformation, and vice versa.
Separable convex functions are multimodular, 
and multimodular functions are integrally convex.
The class of globally discrete midpoint convex functions
lies between the classes of \Lnat-convex functions
and locally discrete midpoint convex functions.
Locally discrete midpoint convex functions are integrally convex.
Jump \Mnat-convex and jump M-convex functions are not necessarily integrally convex.
Jump \Mnat-convex functions contain \Mnat-convex functions 
a special case, and similarly 
jump M-convex functions contain M-convex functions a special case.

The inclusion relations among these function classes
are summarized in the following theorem.
We mention that \eqref{fnfamL} to \eqref{fnfamsepM}
are shown in Fig.~\ref{FGdcfclassZR}.

\begin{theorem} \label{THfnclassinclusion}
The following inclusion relations hold for functions on $\ZZ\sp{n}$:
\begin{align}
&
\{ \mbox{\rm {\rm L}-convex} \} \subsetneqq \ 
\{ \mbox{\rm \Lnat-convex} \} \subsetneqq \ \{ \mbox{\rm integrally convex} \},
\label{fnfamL}
\\ &
\{ \mbox{\rm {\rm M}-convex} \} \subsetneqq \ 
\{ \mbox{\rm \Mnat-convex} \}
\subsetneqq \ \{ \mbox{\rm integrally convex} \},
\label{fnfamM}
\\ &
\{ \mbox{\rm separable convex} \} = \{ \mbox{\rm \Lnat-convex} \} \cap  \{ \mbox{\rm \Mnat-convex} \},
\label{fnfamsepLnMn}
\\ &
\{ \mbox{\rm separable convex} \}  \cap  \{ \mbox{\rm L-convex} \}
 = \{ \mbox{\rm linear on $\ZZ\sp{n}$} \}, 
\label{fnfamsepL}
\\ &
\{ \mbox{\rm separable convex} \}  \cap  \{ \mbox{\rm M-convex} \}
 = \{ \mbox{\rm singleton effective domain} \}, 
\label{fnfamsepM}
\\ &
\{ \mbox{\rm separable convex} \}
\subsetneqq \ \{ \mbox{\rm multimodular} \}
\subsetneqq \ \{ \mbox{\rm integrally convex} \} ,
\label{fnfammultm}
\\ &
\{ \mbox{\rm \Lnat-convex} \}
\subsetneqq \ \{ \mbox{\rm globally d.m.c.} \}
\subsetneqq \ \{ \mbox{\rm locally d.m.c.} \}
\subsetneqq \ \{ \mbox{\rm integrally convex} \},
\label{fnfamdmc}
\\ &
\{ \mbox{\rm M-convex} \}  \subsetneqq \
  \left\{  \begin{array}{l}  \{ \mbox{\rm \Mnat-convex} \}
                          \\ \{ \mbox{\rm jump M-convex} \} \end{array}  \right \}
\subsetneqq \  
\{ \mbox{\rm jump \Mnat-convex} \} 
\not\subseteq \ \{ \mbox{\rm integrally convex} \}.
\label{fnfamjump}
\end{align}
\vspace{-1.3\baselineskip}\\
\finbox
\end{theorem}

\begin{remark} \rm \label{RMfnclassbib}
Here is a supplement to Theorem~\ref{THfnclassinclusion}.
The relation \eqref{fnfamsepLnMn} originates in \cite[Theorem 3.17]{MS01rel}
and is stated in \cite[Theorem 8.49]{Mdcasiam}.
Integral convexity is established for 
\Lnat-convex functions in 
\cite[Theorem 7.20]{Mdcasiam},
for \Mnat-convex functions in 
\cite[Theorem 3.9]{MS01rel}
(see also \cite[Theorem 6.42]{Mdcasiam}), and
for locally (hence globally) discrete midpoint convex functions in
\cite[Theorem 6]{MMTT19dmpc}.
The integral convexity of multimodular functions in \eqref{fnfammultm} was
pointed out first in \cite[Section 14.6]{Mdcaprimer07},
while this is implicit in the construction of the convex extension given earlier in 
\cite[Theorem 4.3]{Haj85}.
The first inclusion in \eqref{fnfammultm} for multimodular functions 
is given in 
\cite[Proposition 2]{MM19multm}.
The inclusion relations in \eqref{fnfamdmc} are given in \cite[Theorem 6]{MMTT19dmpc}.
\finbox
\end{remark}

\begin{remark} \rm \label{RMfndim2}
For functions in two variables,
\Lnat-convexity and \Mnat-convexity are essentially the same 
in the sense that 
a function $f: \ZZ\sp{2} \to \RR \cup \{ +\infty \}$
is \Lnat-convex if and only if
$g(x_{1}, x_{2}) = f(x_{1}, -x_{2})$
is \Mnat-convex.
Moreover, multimodularity is the same as \Mnat-convexity,
that is, a function in two variables
is multimodular if and only if it is \Mnat-convex.
\finbox
\end{remark}

\paragraph*{Convex-extensibility}

Integrally convex functions are convex-extensible (by definition),
and therefore,
separable convex,
L-convex, \Lnat-convex, M-convex, \Mnat-convex, multimodular, and 
globally/locally discrete midpoint convex functions
are convex-extensible.
In contrast, jump \Mnat- or M-convex functions
are not necessarily convex-extensible.
The convex extension of an \Lnat -convex function
can be given by a collection of the (local) Lov{\'a}sz extensions 
(Choquet integrals)
\cite[Section 7.7]{Mdcasiam}.
The convex extension of a multimodular function can be 
constructed in an explicit manner \cite[Theorem 4.3]{Haj85}
(also \cite[Theorem 2.1]{AGH00}),
which may be regarded as a variant of the Lov{\'a}sz extension.
No explicit expression is available for 
the convex extension of an \Mnat- or M-convex function 
\cite[Section 6.10]{Mdcasiam}.

\paragraph*{Set function}

A function $f: 2\sp{N} \to \RR \cup \{ +\infty \}$
that assigns a real number (or $ +\infty$) to each subset of 
$N = \{ 1,2, \ldots, n \}$
is called a {\em set function}.
A set function $f$ is said to be {\em submodular}
\cite{Edm70,Fuj05book,Lov83,Top98}
if 
\begin{equation} \label{submdef=feat}
 f(X) + f(Y) \geq f(X \cup Y) + f(X \cap Y)
 \qquad (X,Y \subseteq N) ,
\end{equation}
where it is understood that the inequality is satisfied
if $f(X)$ or $f(Y)$ is equal to $+\infty$.

A set function $f: 2\sp{N} \to \RR \cup \{ +\infty \}$
can be identified with 
a function $g: \ZZ\sp{n} \to \RR \cup \{ +\infty \}$ with
$\dom g \subseteq \{ 0,1 \}\sp{n}$
through the correspondence
$f(X) = g(\unitvec{X})$ for $X \subseteq N$.
With this correspondence in mind we can say that 
submodular set functions are exactly 
\Lnat-convex functions on $\{ 0,1 \}\sp{n}$,
and {\em valuated matroids}
are exactly M-concave functions on $\{ 0,1 \}\sp{n}$.
\Mnat-concave functions are known to be submodular
\cite[Theorem 3.8]{MS01rel}
(see also \cite[Theorem 6.19]{Mdcasiam}).

\begin{theorem} \label{THsetfnclassinclusion}
The following inclusion relations hold for set functions:
\begin{align}
&
\{ \mbox{\rm valuated matroid} \}
= \{ \mbox{\rm M-concave} \}, 
\\ &
\{ \mbox{\rm \Mnat-concave} \} \subsetneqq \ 
\{ \mbox{\rm submodular} \}
= \{ \mbox{\rm \Lnat-convex} \}.
\end{align}
\vspace{-1.3\baselineskip}\\
\finbox
\end{theorem}





\section{Operations on Discrete Convex Sets}
\label{SCsetope}

\subsection{Operations via simple coordinate changes}
\label{SCchangevarset}

In this section we consider operations on discrete convex sets 
defined by changes of variables such as origin shift,
coordinate inversion, permutation of variables, and scaling of variables.

Let $S$ be a subset of $\ZZ\sp{n}$, i.e.,  $S \subseteq \ZZ\sp{n}$.
For an integer vector $b \in \ZZ\sp{n}$, the
{\em origin shift}
of $S$ by $b$
means a subset $T$ of $\ZZ\sp{n}$ defined by%
\begin{align}  \label{shiftsetdef}
 T &  = \{  x - b \mid x \in S \} .
\end{align}

\begin{proposition} \label{PRsetshift}
The origin shift operation \eqref{shiftsetdef} for a set preserves
integral convexity,
L$^{\natural}$-convexity, L-convexity,
M$^{\natural}$-convexity, M-convexity,
multimodularity, and
discrete midpoint convexity.
Moreover, the origin shift of an integer box is an integer box,
and the origin shift of 
an s.e. (resp., c.p.) jump system
is an s.e. (resp., c.p.) jump system.
\finbox
\end{proposition}

For coordinate inversions, 
two types are distinguished.
The {\em simultaneous coordinate inversion} for $S$
means a subset $T$ of $\ZZ\sp{n}$ defined by
\begin{align} 
T  &= 
\{ (-x_{1},-x_{2}, \ldots, -x_{n}) \mid (x_{1},x_{2}, \ldots, x_{n}) \in S \},
\label{signinvsetdef}
\end{align}
and the {\em independent coordinate inversion} for $S$ 
means a subset $T$ of $\ZZ\sp{n}$ defined by
\begin{align}
T  &=
\{ (\tau_{1} x_{1}, \tau_{2} x_{2}, \ldots,  \tau_{n}x_{n}) \mid (x_{1},x_{2}, \ldots, x_{n}) \in S \}
 \label{indepsigninvsetdef}
\end{align}
with an arbitrary choice of $\tau_{i} \in \{ +1, -1 \}$ $(i=1,2,\ldots,n)$.

\begin{proposition} \label{PRsetsigninv}
The simultaneous coordinate inversion operation \eqref{signinvsetdef} for a set preserves
integral convexity,
L$^{\natural}$-convexity, L-convexity,
M$^{\natural}$-convexity, M-convexity,
multimodularity, and
discrete midpoint convexity.
Moreover,
the simultaneous coordinate inversion of an integer box is an integer box,
and the simultaneous coordinate inversion of
an s.e. (resp., c.p.) jump system is an s.e. (resp., c.p.) jump system.
\finbox
\end{proposition}

\begin{proposition} \label{PRsetindepsigninv}
The independent coordinate inversion operation \eqref{indepsigninvsetdef} for a set preserves
integral convexity.
Moreover,
the independent coordinate inversion of an integer box is an integer box,
and the independent coordinate inversion of an s.e. (resp., c.p.) jump system
is an s.e. (resp., c.p.) jump system.
\finbox
\end{proposition}

The independent coordinate inversion operation 
does not preserve
L$^{\natural}$-convexity, L-convexity,
M$^{\natural}$-convexity, M-convexity,
multimodularity, or
discrete midpoint convexity.

\begin{example}  \rm \label{EXlsetsigninv} 
Let $S = \{ (t,t) \mid t \in \ZZ \}$ and $T = \{ (t,-t) \mid t \in \ZZ \}$,
which is obtained from $S$ by an independent coordinate inversion 
\eqref{indepsigninvsetdef} with $\tau_{1}=+1$ and $\tau_{2}=-1$.
The set $S$ is L-convex,  and hence  \Lnat -convex
and discrete midpoint convex,
whereas $T$ is not L-convex,  \Lnat -convex, or discrete midpoint convex.
\finbox
\end{example}

\begin{example}  \rm \label{EXmsetsigninv} 
Let $S = \{ (t,-t) \mid t \in \ZZ \}$ and $T = \{ (t,t) \mid t \in \ZZ \}$,
which is obtained from $S$ by an independent coordinate inversion 
\eqref{indepsigninvsetdef} with $\tau_{1}=+1$ and $\tau_{2}=-1$.
The set $S$ is M-convex  and hence  \Mnat -convex,
whereas $T$ is not M-convex or \Mnat -convex.
The set $S$ is multimodular, and $T$ is not.
\finbox
\end{example}

For a permutation $\sigma$ of $(1,2,\ldots,n)$,
the {\em permutation} of $S$ by $\sigma$
means a subset $T$ of $\ZZ\sp{n}$ defined by
\begin{align}
T &=   \{ (y_{1},y_{2}, \ldots, y_{n})   
 \mid (y_{\sigma(1)}, y_{\sigma(2)}, \ldots, y_{\sigma(n)}) \in S \}.
 \label{permsetdef}
\end{align}

\begin{proposition} \label{PRsetperm}
The permutation operation \eqref{permsetdef} for a set preserves
integral convexity,
L$^{\natural}$-convexity, L-convexity,
M$^{\natural}$-convexity, M-convexity,
and
discrete midpoint convexity.
Moreover,
the permutation of an integer box is an integer box,
and the permutation of  
an s.e. (resp., c.p.) jump system is an s.e. (resp., c.p.) jump system.
\finbox
\end{proposition}

The permutation operation does not preserve multimodularity.

\begin{example}  \rm \label{EXmmsetperm} 
$S = \{ (0,0,0), (0,1,-1), (0,1,0), (1,0,0) \}$
is a multimodular set, but the set  
$T = \{ (0,0,0), (1,0,-1), (1,0,0), (0,1,0) \}$
obtained from $S$ by a permutation (transposition) 
$\sigma: (1,2,3) \mapsto (2,1,3)$ is not multimodular.
Indeed, the transformed set
\[
\tilde S = \{ D\sp{-1} x \mid x \in S \} = \{ (0,0,0), (0,1,0), (0,1,1), (1,1,1) \} 
\]
is \Lnat-convex, whereas
\[
\tilde T = \{ D\sp{-1} x \mid x \in T \} = \{ (0,0,0), (1,1,0), (1,1,1), (0,1,1) \}
\]
is not \Lnat-convex, since
$\lfloor  ( (1,1,0)+(0,1,1)) /2 \rfloor = (0,1,0)$
does not belong to $\tilde T$.
\finbox
\end{example}

For a positive integer $\alpha$,
the {\em scaling} of $S$ by $\alpha$
means a subset $T$ of $\ZZ\sp{n}$ defined by
\begin{equation} \label{scalesetdef}
T = \{ (y_{1},y_{2}, \ldots, y_{n}) \in \ZZ\sp{n}
 \mid (\alpha y_{1}, \alpha y_{2}, \ldots, \alpha y_{n}) \in S \}.
\end{equation}
Note that the same scaling factor $\alpha$ is used for all coordinates.

L-convexity and its relatives is well-behaved 
with respect to the scaling operation.%
\footnote{
The scaled set $T$ can be empty even when $S$ is nonempty.
Therefore, strictly speaking, we should 
add a proviso in Propositions \ref{PRsetscale} and \ref{PRsetscdim2}
that the resulting set is nonempty.
}

\begin{proposition} \label{PRsetscale}
The scaling operation \eqref{scalesetdef} for a set preserves
L$^{\natural}$-convexity, L-convexity,
multimodularity,  and discrete midpoint convexity.
Moreover, the scaling of an integer box is an integer box.
\finbox
\end{proposition}

\begin{remark} \rm \label{RMsetscalebib}
Here is a supplement to Proposition \ref{PRsetscale} about scaling.
The statements for L$^{\natural}$-convex and L-convex sets
are easy to prove and well known.
The statement for multimodular sets
is given in \cite[Proposition 7]{MM19multm} 
and that for discrete midpoint convex sets
in \cite[Proposition 9]{MMTT19dmpc}.
The statements given in Proposition \ref{PRsetscale}
follow from the corresponding statements for functions 
in Proposition \ref{PRfnscale}.
\finbox
\end{remark}

In contrast, M-convexity and integral convexity
are not compatible with the scaling operation, as follows.

\begin{itemize}

\item
The scaling of an \Mnat-convex set
is not necessarily \Mnat-convex.
See Example~\ref{EXmnatsetscdim3}.

\item
The scaling of an M-convex set
is not necessarily M-convex.
See Example~\ref{EXmsetscdim3}.

\item
The scaling of an s.e. (resp., c.p.) jump system
is not necessarily 
an s.e. (resp., c.p.) jump system.
See Examples \ref{EXmnatsetscdim3} and \ref{EXmsetscdim3}.

\item
The scaling of an integrally convex set
is not necessarily integrally convex.
See Example~\ref{EXscICsetNG422}.
\end{itemize}

\begin{example}[{\cite[Example 1.1]{MMTT19proxIC}}] \rm \label{EXmnatsetscdim3}
This example shows that ${\rm M}^{\natural}$-convexity is not preserved under scaling.
Let $S$ be a subset of $\ZZ^{3}$ defined as
\[
S =  \{ c_{1} (1,0,-1) +c_{2} (1,0,0) + c_{3} (0,1,-1) + c_{4} (0,1,0)
      \mid c_{i} \in \{ 0,1 \} \ (i=1,2,3,4)  \} .
\]
This is an ${\rm M}^{\natural}$-convex set,
but the scaled set
$T = \{ y \in \ZZ\sp{3} \mid 2 y \in S \} =  \{ (0,0,0), (1,1,-1) \}$
(with the scaling factor $\alpha=2$)
is not ${\rm M}^{\natural}$-convex.
This example also shows that 
the scaling of an s.e.~jump system
is not necessarily 
an s.e.~jump system.
\finbox
\end{example}

\begin{example}[{\cite[Note 6.18]{Mdcasiam}}] \rm \label{EXmsetscdim3}
This example shows that {\rm M}-convexity is not preserved under scaling.
Let $S$ be a subset of $\ZZ^{4}$ defined as
\begin{align*}
S =& \{ c_{1} (1,0,-1,0) +c_{2} (1,0,0,-1) + c_{3} (0,1,-1,0) + c_{4} (0,1,0,-1) \mid
\\ & \phantom{ \{ }
 c_{i} \in \{ 0,1 \} \ (i=1,2,3,4)  \} .
\end{align*}
This is an {\rm M}-convex set,
but the scaled set
$T = \{ y \in \ZZ\sp{4} \mid 2 y \in S \}  =  \{ (0,0,0,0), (1,1,-1,-1) \}$
(with the scaling factor $\alpha=2$)
is not {\rm M}-convex.
This example also shows that 
the scaling 
of a c.p.~jump system
is not necessarily 
a c.p.~jump system.
\finbox
\end{example}

\begin{figure}\begin{center}
\includegraphics[height=45mm]{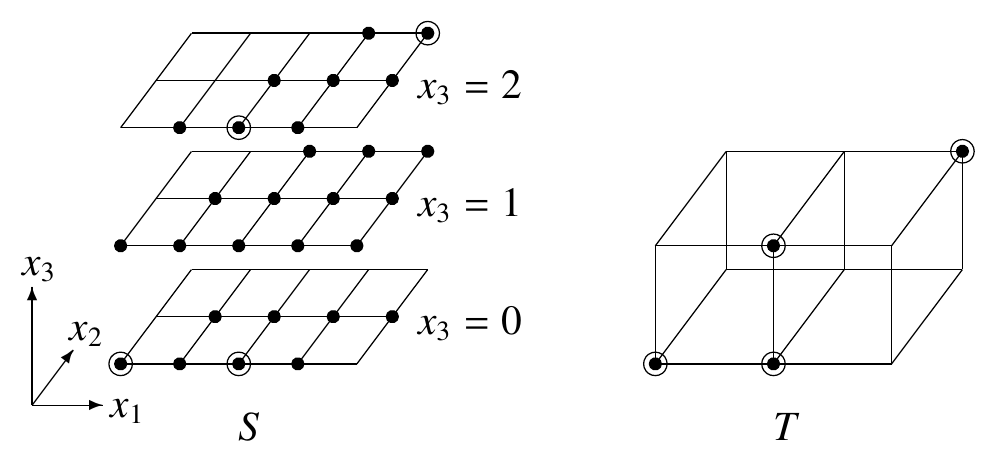}
\caption{An integrally convex set $S$ and its scaled set $T$ (Example \ref{EXscICsetNG422})}
\label{FGicsetsc}
\end{center}\end{figure}

\begin{example}[{\cite[Example 3.1]{MMTT19proxIC}}] \rm \label{EXscICsetNG422}
This example shows that integral convexity is not preserved under scaling.
Let $S$ be a subset of $\ZZ^{3}$ defined  by
\begin{align*}
S  = & 
\{ (x_{1},x_{2},0) \mid 
0 \leq x_{2} \leq 1,  \  0 \leq x_{1} - x_{2} \leq 3 
\}  
\\ & \cup
\{ (x_{1},x_{2},1) \mid 
0 \leq x_{2} \leq 2, \  x_{2} \leq x_{1} \leq 4
\}  
\\ & \cup
\{ (x_{1},x_{2},2) \mid 
0 \leq x_{2} \leq 2, \ 1 \leq x_{1} - x_{2} \leq 3, \  x_{1} \leq 4
\}  ,
\end{align*}
which is an integrally convex set (cf., Fig.~\ref{FGicsetsc}, left).
With the scaling factor $\alpha=2$,
however, 
the scaled set
\[
 T = \{ y \in \ZZ\sp{3} \mid 2 y \in S \} = \{ (0,0,0), (1,0,0), (1,0,1), (2,1,1) \}
\]
is not integrally convex  
(cf., Fig.~\ref{FGicsetsc}, right).
\finbox
\end{example}

In the case of $n = 2$,
M-convexity and integral convexity admit the scaling operation.

\begin{proposition}  \label{PRsetscdim2}
Let $\alpha$ be a positive integer,
$S \subseteq \ZZ^{2}$, 
and  $T = \{ y \in \ZZ^{2} \mid \alpha y \in S \}$.


\noindent
{\rm (1)}
If $S$ is integrally convex, then $T$ is integrally convex.

\noindent
{\rm (2)}
If $S$ is \Mnat-convex, then $T$ is \Mnat-convex.

\noindent
{\rm (3)}
If $S$ is M-convex, then $T$ is M-convex.
\finbox
\end{proposition}

\begin{remark} \rm \label{RMfnscdim2}
Here is a supplement to Proposition \ref{PRsetscdim2} about scaling for two-dimensional sets.
Part (1) for integrally convex sets is due to \cite[Theorem 3.2]{MMTT19proxIC}.
Part (2) for \Mnat-convex sets follows from 
the statement for \Lnat-convex sets in
Proposition \ref{PRsetscale} and Remark \ref{RMsetdim2} in Section~\ref{SCclasssetZ},
while Part (3) for M-convex sets is almost a triviality.
\finbox
\end{remark}

Table \ref{TBoperation1dcsetZ}
is a summary of the behavior of discrete convex sets 
with respect to the simple coordinate changes considered above.
In this table,
``Y'' means ``Yes, this set class is closed under this operation''
and  ``\textbf{\textit{N}}'' means ``No, this set class is not closed under this operation.''

\begin{table}
\begin{center}
\caption{Operations on discrete convex sets via simple coordinate changes}
\label{TBoperation1dcsetZ}

\medskip


\addtolength{\tabcolsep}{-3pt}
\begin{tabular}{l|c|cc|c|c|l}
 Discrete convex set & Origin & \multicolumn{2}{c|}{Coord.~inversion}  & Permu- & Scaling & Reference
\\ \cline{3-4}
 & shift &  simult. & indep. &  tation & 
\\ \hline
 Integer box  & \YES & \YES & \YES & \YES & \YES  & 
\\
 Integrally convex & \YES & \YES & \YES  & \YES & \NO
   & \cite{FT90,MMTT19proxIC,Mdcasiam}
\\
\Lnat-convex       & \YES & \YES & \NO & \YES & \YES 
  & \cite{FM00,Mdcasiam} 
\\ 
L-convex           & \YES & \YES & \NO & \YES & \YES 
  & \cite{Mdca98,Mdcasiam}
\\ 
\Mnat-convex       & \YES & \YES & \NO& \YES & \NO
  & \cite{Mdcasiam,MS99gp}
\\ 
M-convex           & \YES & \YES & \NO & \YES & \NO
  & \cite{Mstein96,Mdcasiam}
\\ 
Multimodular       & \YES & \YES & \NO  &  \NO  & \YES
  & \cite{AGH00,AGH03,Haj85,MM19multm}
\\ 
Disc.~midpt convex   & \YES & \YES & \NO & \YES & \YES 
  & \cite{MMTT19dmpc}
\\ 
Simul.~exch. jump  & \YES & \YES & \YES & \YES & \NO  
  &  \cite{Mmnatjump19} 
\\ 
Const-parity jump   & \YES & \YES & \YES & \YES & \NO  
  &  \cite{BouC95} 
\\ \hline
\multicolumn{7}{l}{``Y'' means ``Yes, this set class is closed under this operation.''} 
\\
\multicolumn{7}{l}{``\textbf{\textit{N}}'' means ``No, this set class is not closed under this operation.''} 
\end{tabular}
\addtolength{\tabcolsep}{3pt}
\end{center}
\end{table}

\subsection{Restriction}
\label{SCrestrset}

For a set $S \subseteq \ZZ\sp{N}$
and a subset $U$ of the index set $N = \{ 1,2,\ldots, n \}$,
the {\em restriction} 
of $S$ to $U$ is a subset $T$ of $\ZZ\sp{U}$ defined by%
\footnote{
More generally, we can define the restriction of $S$ to $U$ as
$T =  \{ y \in \ZZ\sp{U} \mid  (y,z) \in S \}$
using an arbitrary $z\in \ZZ\sp{N \setminus U}$.
For simplicity of description we here choose $z=\veczero_{N \setminus U}$.
} 
\begin{equation} \label{restrsetdef}
 T =  \{ y \in \ZZ\sp{U} \mid  (y,\veczero_{N \setminus U}) \in S \}, 
\end{equation}
where $\veczero_{N \setminus U}$ denotes the zero vector 
in $\ZZ\sp{N \setminus U}$.
The notation 
$(y,\veczero_{N \setminus U})$ means the vector in $\ZZ\sp{N}$
whose $i$th component is equal to $y_{i}$ for $i \in U$
and to $0$ for $i \in N \setminus U$;
for example, if $N = \{ 1,2,3 \}$ and $U = \{ 1,3 \}$,
then
$(y,\veczero_{N \setminus U})$ means $(y_{1}, 0, y_{3})$.

The restriction operation preserves discrete convexity as follows.%
\footnote{
The restriction operation may result in an empty set.
Therefore, strictly speaking, we should 
add a proviso in Proposition \ref{PRsetrestr}
that the resulting set is nonempty.
}

\begin{proposition} \label{PRsetrestr}
The restriction operation \eqref{restrsetdef} for a set preserves
integral convexity,
L$^{\natural}$-convexity, 
M$^{\natural}$-convexity, M-convexity,
multimodularity, and
discrete midpoint convexity.
Moreover,
the restriction of an integer box is an integer box,
and the restriction 
of an s.e. (resp., c.p.) jump system is an s.e. (resp., c.p.) jump system.
\finbox
\end{proposition}

\begin{remark} \rm \label{RMsetrestrbib}
Here is a supplement to Proposition \ref{PRsetrestr} about restriction. 
The statements given in Proposition \ref{PRsetrestr}
follow from the corresponding statements for functions 
in Proposition \ref{PRfnrestr}.
The statement for multimodular sets is given in \cite[Proposition 9]{MM19multm}. 
The statements for (s.e. or c.p.) jump systems are immediate from  \cite[Section 3]{BouC95}.
\finbox
\end{remark}

The restriction of an L-convex set is not necessarily L-convex,
simply because the restricted set lacks in
the translation invariance 
in the direction of $\vecone$.

\begin{example}  \rm \label{EXlsetrestr} 
$S = \{ (t,t) \mid t \in \ZZ \}$
is an L-convex set in $\ZZ\sp{2}$ and its restriction to $U = \{ 1 \}$ is
$T = \{ 0 \}$, which is not L-convex.
\finbox
\end{example}

Table~\ref{TBoperation2dcsetZ} 
is a summary of the behavior of discrete convex sets 
with respect to the restriction operation as well as
the operations of projection, intersection, and Minkowski sum,
to be discussed below.
Two types of intersections are distinguished:
one is the intersection of $S$ 
with an integer box $B$,
denoted $S \cap B$,
and the other is the intersection  
$S_{1} \cap S_{2}$ 
of two sets $S_{1}$ and $S_{2}$ in the same convexity class. 
Similarly for the Minkowski sum,
denoted $S + B$ and $S_{1} +  S_{2}$.

\begin{table}
\begin{center}
\caption{Operations on discrete convex sets}
\label{TBoperation2dcsetZ}

\medskip

\addtolength{\tabcolsep}{-3pt}
\begin{tabular}{l|c|c|cc|cc|l}
 Discrete  & Restric- & Projec- & \multicolumn{2}{c|}{Intersection} & \multicolumn{2}{c|}{Minkow.~sum} & Reference
 \\ \cline{4-7}
\quad convex set & tion  &  tion  &  box & general &   box & general      
\\ 
   & \footnotesize $\{ y | (y,\bm{0}) \! \in \! S \! \}$ 
   & \footnotesize $\{ y | (y,z) \! \in \! S \! \}$  
   & \footnotesize $S \cap B$  
   & \footnotesize  $S_{1} \cap S_{2}$ 
   & \footnotesize  $S + B$ 
   & \footnotesize  $S_{1}+ S_{2}$
\\ \hline  
 Integer box & \YES & \YES & \YES & \YES & \YES & \YES
\\ 
 Integrally convex & \YES & \YES & \YES & \NO & \YES &\NO
   & \cite{FT90,MM19projcnvl,Mdcasiam}
\\
\Lnat-convex   & \YES & \YES & \YES & \YES & \YES & \NO 
  & \cite{FM00,Mdcasiam} 
\\ 
L-convex  & \NO  & \YES & \NO \quad & \YES & \YES & \NO 
  & \cite{Mdca98,Mdcasiam}
\\
\Mnat-convex  &  \YES & \YES & \YES & \NO & \YES & \YES
  & \cite{Mdcasiam,MS99gp}
\\ 
M-convex  & \YES & \NO & \YES & \NO & \NO & \YES
  & \cite{Mstein96,Mdcasiam}
\\ 
Multimodular  &  \YES  & \NO  & \YES & \YES & \NO & \NO 
  & \cite{AGH00,AGH03,Haj85,MM19multm}
\\ 
Disc.~midpt conv.  &  \YES  & \YES & \YES & \YES 
  & \NO  & \NO 
  & \cite{MM19projcnvl,MMTT19dmpc}
\\ 
Sim.~exch. jump & \YES  & \YES & \YES 
  & \NO & \YES & \YES
  &  \cite{Mmnatjump19}
\\ 
Const-par.~jump & \YES  &  \NO  & \YES & \NO 
  & \NO & \YES
  &  \cite{BouC95} 
\\ \hline
\multicolumn{3}{l}{} &
\multicolumn{4}{r}{($B$: integer box)} 
\\
\multicolumn{7}{l}{``Y'' means ``Yes, this set class is closed under this operation.''} 
\\
\multicolumn{7}{l}{``\textbf{\textit{N}}'' means ``No, this set class is not closed under this operation.''} 
\end{tabular}
\addtolength{\tabcolsep}{3pt}
\end{center}
\end{table}

\subsection{Projection}

For a set $S \subseteq \ZZ\sp{N}$
and a subset $U$ of the index set $N = \{ 1,2,\ldots, n \}$,
the
{\em projection}
of $S$ to $U$ is a subset $T$ of $\ZZ\sp{U}$ defined by
\begin{equation}  \label{projsetdef}
 T =  \{ y \in \ZZ\sp{U} \mid  (y,z) \in S \mbox{ for some $z \in \ZZ\sp{N \setminus U}$} \}, 
\end{equation}
where the notation $(y,z)$ means the vector in $\ZZ \sp{N}$
whose $i$th component is equal to $y_{i}$ for $i \in U$
and to $z_{i}$ for $i \in N \setminus U$;
for example, if $N = \{ 1,2,3,4 \}$ and $U = \{ 2,3 \}$,
we have
$(y,z) = (z_{1}, y_{2}, y_{3}, z_{4})$.

The projection operation preserves discrete convexity as follows.

\begin{proposition} \label{PRsetproj}
The projection operation \eqref{projsetdef} for a set preserves
integral convexity,
L$^{\natural}$-convexity, L-convexity,
M$^{\natural}$-convexity, 
and
discrete midpoint convexity.
Moreover,
the projection of an integer box is an integer box,
and the projection of an s.e.~jump system is an s.e.~jump system.
\finbox
\end{proposition}

\begin{remark} \rm \label{RMsetprojbib}
Here is a supplement to Proposition \ref{PRsetproj} about projection.
The result for integrally convex sets
is established in \cite[Theorem 3.1]{MM19projcnvl}, 
those for L$^{\natural}$-convex, L-convex, and 
M$^{\natural}$-convex sets
follow from 
\cite[Theorem 7.11]{Mdcasiam}, 
\cite[Theorem 7.10]{Mdcasiam}, and
\cite[Theorem 6.15]{Mdcasiam},
respectively.
The result for discrete midpoint convex sets
is proved in \cite[Theorem 3.4]{MM19projcnvl}, and 
that for s.e.~jump systems in \cite{Mmnatjump19}.
The statements given in Proposition \ref{PRsetproj}
follow from the corresponding statements for functions 
in Proposition \ref{PRfnproj}.
\finbox
\end{remark}

The projection of an M-convex set is not necessarily M-convex,
simply because the projected set does not lie on a hyperplane of a constant component sum.
Similarly,
the projection of a constant-parity jump system
is not necessarily a constant-parity jump system.

\begin{example}  \rm \label{EXmsetproj} 
$S = \{ (t,-t) \mid t \in \ZZ \}$
is an M-convex set in $\ZZ\sp{2}$ 
and its projection to $U = \{ 1 \}$ is
$T = \ZZ$, which is not M-convex.
The set $S$ is a constant-parity jump system, while $T$ is not.
\finbox
\end{example}

The projection of a multimodular set is not necessarily multimodular.

\begin{example}   \rm \label{EXmmsetproj}
$S = \{ (0,0,0), (0,1,-1), (0,1,0), (1,0,0) \}$
is a multimodular set (Example~\ref{EXmmsetperm}).
Its projection to $U= \{ 1,3 \}$ is given by 
$T = \{ (0,0), (0,-1), (1,0) \}$,
which is not multimodular, since the transformed set 
$\tilde T = \{ D\sp{-1} x \mid x \in T \} = \{ (0,0), (0,-1), (1,1) \}$
is not \Lnat-convex.
\finbox
\end{example}

A subset $U$ of the index set $N = \{ 1,2,\ldots, n \}$
is said to be {\em consecutive}
if it consists of consecutive numbers, that is,
it is a set of the form 
$\{ k, k+1, \ldots, l -1, l \}$
for some $k \leq l$.
The projection of a multimodular set to a consecutive index subset $U$ is multimodular.

\begin{proposition}[{\cite[Proposition 10]{MM19multm}}]  \label{PRmmsetprojinterval}
The projection of a multimodular set $S$ 
to a consecutive index subset $U$ is multimodular.
\finbox
\end{proposition}

\subsection{Intersection}

We consider discrete convexity of
the intersection $S_{1} \cap S_{2}$ 
of two discrete convex sets 
$S_{1}$, $S_{2} \subseteq \ZZ\sp{n}$.
L-convexity and its relatives are well-behaved 
with respect to the intersection operation.%
\footnote{
The intersection operation may result in an empty set.
Therefore, strictly speaking, we should 
add a proviso in Propositions 
\ref{PRlsetinter}, \ref{PRmsetinter}, and \ref{PRsetboxinter}
that the resulting set is nonempty.
}

\begin{proposition} \label{PRlsetinter}
\quad


\noindent
{\rm (1)} 
The intersection of integer boxes is an integer box.

\noindent
{\rm (2)} 
The intersection of \Lnat-convex sets is \Lnat-convex.

\noindent
{\rm (3)} 
The intersection of L-convex sets is L-convex.

\noindent
{\rm (4)} 
The intersection of multimodular sets is multimodular.

\noindent
{\rm (5)} 
The intersection of discrete midpoint convex sets
is discrete midpoint convex.
\finbox
\end{proposition}

\begin{remark} \rm \label{RMsetinterbib}
Here is a supplement to Proposition \ref{PRlsetinter} about intersection.
This property for L-convex sets is given in \cite[Theorem 5.7]{Mdcasiam}, 
whereas that for L$^{\natural}$-convex sets is an immediate corollary thereof.
This property for multimodular sets is stated in \cite[Proposition 8]{MM19multm},
and that for discrete midpoint convex sets in \cite[Proposition 2]{MMTT19dmpc}.
\finbox
\end{remark}

In contrast, M-convexity and integral convexity
are not compatible with the intersection operation.

\begin{itemize}
\item
The intersection of integrally convex sets
is not necessarily integrally convex.
See Example~\ref{EXicsetinter}.

\item
The intersection of \Mnat-convex sets
is not necessarily \Mnat-convex.
See Example~\ref{EXmnatsetinter}.

\item
The intersection of M-convex sets
is not necessarily M-convex.
See Example~\ref{EXmsetinter}.

\item
The intersection 
of s.e. (resp., c.p.) jump systems
is not necessarily 
an s.e. (resp., c.p.) jump system.
See Example~\ref{EXmsetinter}.

\end{itemize}

\begin{example}[{\cite[Example 4.4]{MS01rel}}] \rm \label{EXicsetinter}
The intersection of two integrally convex sets
is not necessarily integrally convex. 
Let
\[
S_1  =  \{(0, 0, 0), (0, 1, 1), (1, 1, 0), (1, 2, 1)\},
\quad
S_2  =  \{(0, 0, 0),  \allowbreak  (0, 1, 0), (1, 1, 1), (1, 2, 1)\},
\]
for which
$S_1 \cap S_2= \{(0, 0, 0), (1, 2, 1)\}$.
The sets $S_1$ and $S_2$ are both integrally convex,
whereas  $S_1 \cap S_2$ is not integrally convex.
\finbox
\end{example}

\begin{example}[{\cite[Example 3.7]{MS01rel}}] \rm \label{EXmnatsetinter}
The intersection of two \Mnat -convex sets is not necessarily \Mnat -convex. 
Let
\begin{align*}
S_{0} &= \{(0, 0, 0), (1, 0, 0), (0, 1, 0),  (0, 0, 1), (1, 0, 1)\},
\\
S_{1} &= S_{0} \cup \{(0, 1, 1)\}, 
\qquad
S_{2} = S_{0} \cup \{(1, 1, 0)\}.
\end{align*}
Here $S_{1}$ and $S_{2}$ are M$\sp{\natural}$-convex sets,
whereas $S_{0} = S_{1} \cap S_{2}$ is not.
Indeed, the exchange property 
\Bnvex
for \Mnat-convex sets
fails for $S_{0}$ with
$x = (1, 0, 1)$, $y = (0, 1, 0)$, and $i = 1 \in \suppp(x-y)$.
\finbox
\end{example}

\begin{example}[{\cite[Note 4.25]{Mdcasiam}}] \rm \label{EXmsetinter}
The intersection of two M-convex sets is not necessarily M-convex. 
Let
\begin{align*}
S_{0} & = \{(0, 0, 0,0), (1, 0, 0,-1), (0, 1, 0,-1), (0, 0, 1,-1), (1, 0, 1,-2)\},
\\
S_{1} &= S_{0} \cup \{(0, 1, 1,-2)\}, 
\qquad
S_{2} = S_{0} \cup \{(1, 1, 0,-2)\}.
\end{align*}
Here $S_{1}$ and $S_{2}$ are M-convex,
whereas $S_{0} = S_{1} \cap S_{2}$ is not.
Indeed,
the exchange property 
\Bvex 
for M-convex sets
fails for $S_{0}$ with
$x = (1, 0, 1,-2)$, $y = (0, 1, 0,-1)$, and $i = 1 \in \suppp(x-y)$.

This example also shows that 
the intersection of two 
(s.e.~or c.p.) 
jump systems is not necessarily a jump system,
since a constant-sum system
is a jump system if and only if it is an M-convex set.
\finbox
\end{example}

The intersection of two M$^{\natural}$-convex sets is integrally convex,
though not M$^{\natural}$-convex.

\begin{proposition}[{\cite[Theorem 8.31]{Mdcasiam}}]  \label{PRmsetinter}
The intersection of two M$^{\natural}$-convex sets is integrally convex.
In particular, the intersection of two M-convex sets is integrally convex.
\finbox
\end{proposition}

\begin{example} \rm \label{EXlnatsetMinter2}
In Example~\ref{EXmsetinter},
the intersection $S_{0} = S_{1} \cap S_{2}$ of two M-convex sets $S_{1}$ and $S_{2}$
is not M-convex, but it is integrally convex.
For $x = (0, 0, 0, 0)$ and $y = (1,0,1,-2)$ in $S_{0}$
with $\| x - y \|_{\infty} =2$, 
for example,
the midpoint
$z = (x+y)/2 = (1/2, 0, 1/2, -1)$
can be expressed as 
$z = ((1,0, 0,-1) + (0,0,1,-1))/2$
using two points
$(1,0, 0,-1)$ and $(0,0,1,-1)$
in $S_{0} \cap N(z)$.
\finbox
\end{example}

\begin{remark} \rm \label{RMdcnvhull}
Proposition~\ref{PRlsetinter} enables us to define
the ``\Lnat-convex hull'' of a set $S \subseteq \ZZ\sp{n}$.
Indeed, the intersection 
of all \Lnat-convex sets containing $S$
is the smallest \Lnat-convex set containing $S$,
which we can naturally call the ``\Lnat-convex hull'' of $S$.
In a similar manner we can define
the ``L-convex hull'' of $S$ and 
the ``discrete midpoint convex hull'' of $S$.
However, such definition does not work for
\Mnat-convexity, M-convexity, or integral convexity.
\finbox
\end{remark}

We next consider 
the intersection
of a set $S \subseteq \ZZ\sp{n}$ with an integer box $B$.
If $S$ is equipped with a certain discrete convexity, 
the intersection $S \cap B$ also possesses the same kind of
discrete convexity, as follows.

\begin{proposition} \label{PRsetboxinter}
\quad


\noindent
{\rm (1)} 
The intersection of an integrally convex set
with an integer box is integrally convex.

\noindent
{\rm (2)} 
The intersection of an \Lnat-convex set 
with an integer box is \Lnat-convex.

\noindent
{\rm (3)} 
The intersection of an \Mnat-convex set 
with an integer box is \Mnat-convex.

\noindent
{\rm (4)} 
The intersection of an M-convex set 
with an integer box is M-convex.

\noindent
{\rm (5)} 
The intersection of a multimodular set 
with an integer box is multimodular.

\noindent
{\rm (6)} 
The intersection of a discrete midpoint convex set
with an integer box is discrete midpoint convex.

\noindent
{\rm (7)} 
The intersection of an s.e.~jump system
with an integer box is an s.e.~jump system.

\noindent
{\rm (8)} 
The intersection of a c.p.~jump system
with an integer box is a c.p.~jump system.
\finbox
\end{proposition}

\begin{remark} \rm \label{RMsetboxinterbib}
Here is a supplement to Proposition \ref{PRsetboxinter}
about the intersection with an integer box.
The statements (1)--(8) are immediate from the definitions.
Parts (2), (5), and (6)
for \Lnat-convex, multimodular, and discrete midpoint convex sets
  are special cases 
of Proposition~\ref{PRlsetinter} (2), (4), and (5),
respectively, since an integer box is 
\Lnat-convex, multimodular, and discrete midpoint convex
(cf., Theorem~\ref{THsetclassinclusion}).
\finbox
\end{remark}

In this connection we note the following fact.

\begin{itemize}
\item
The intersection of an L-convex set with an integer box is 
not necessarily L-convex.
This is because the intersected set lacks in
the translation invariance in the direction of $\vecone$.
See Example~\ref{EXlsetboxinter}.
\end{itemize}

\begin{example}  \rm \label{EXlsetboxinter} 
Let $S = \{ (t,t) \mid t \in \ZZ \}$
and $B = [ \veczero, \vecone ]_{\ZZ}$.
$S$ is an L-convex set and $B$ is an integer box.
Their intersection $S \cap B$ is equal to 
$\{ (0,0), (1,1) \}$, which is not L-convex.
\finbox
\end{example}

\subsection{Minkowski sum}

The
{\em Minkowski sum}
of two sets $S_{1}$, $S_{2} \subseteq \ZZ\sp{n}$ 
means the subset of $\ZZ\sp{n}$ 
defined by
\begin{equation} \label{minkowsumZdef}
S_{1}+S_{2} = 
\{ x + y \mid x \in S_{1}, \  y \in S_{2} \} .
\end{equation}
We sometimes refer to this
as the discrete (or integral) Minkowski sum
to emphasize discreteness.
The Minkowski sum is often a source of difficulty in a discrete setting,
since 
\[
 ( \overline{S_{1}+ S_{2}}) \cap \ZZ\sp{n} \not=   S_{1}+S_{2}
\]
in general,
as is  demonstrated by Example~\ref{EXicdim2sumhole} below.

\begin{example}[{\cite[Example 3.15]{Mdcasiam}}] \rm \label{EXicdim2sumhole}
The Minkowski sum of
$S_{1} = \{ (0,0), (1,1) \}$
and
$S_{2} = \{ (1,0), (0,1) \}$
is equal to 
\[
S_{1}+S_{2} = \{ (1,0), (0,1), (2,1), (1,2) \},
\]
which has a ``hole'' at $(1,1)$, i.e.,
$(1,1) \in \overline{S_{1}+S_{2}}$ and
$(1,1) \not\in S_{1}+S_{2}$.
\finbox
\end{example}

M-convexity  and its relatives are well-behaved 
with respect to the Minkowski sum.

\begin{theorem} \label{THmsetMinkow}
\quad


\noindent
{\rm (1)} 
The Minkowski sum of integer boxes is an integer box.

\noindent
{\rm (2)} 
The Minkowski sum of \Mnat-convex sets is \Mnat-convex.

\noindent
{\rm (3)} 
The Minkowski sum of M-convex sets is M-convex.

\noindent
{\rm (4)} 
The Minkowski sum of s.e.~jump systems is an s.e.~jump system.

\noindent
{\rm (5)} 
The Minkowski sum of c.p.~jump systems is a c.p.~jump system.
\finbox
\end{theorem}

\begin{remark} \rm \label{RMsetMsumbib}
Here is a supplement to Theorem~\ref{THmsetMinkow}
about the Minkowski sum.
The Minkowski sum of integer boxes 
$[a_{1},b_{1}]_{\ZZ}$ and
$[a_{2},b_{2}]_{\ZZ}$ 
is equal to the integer box
$[a_{1}+ a_{2} ,b_{1} + b_{2}]_{\ZZ}$.
Part (3) for M-convex sets is 
a translation of a known fact in submodular function theory
and is given explicitly in \cite[Theorem 4.23]{Mdcasiam}.
Part (2) for \Mnat-convex sets is easily obtained as a variant of (3) for M-convex sets,
or as a special case of \cite[Theorem 6.15]{Mdcasiam}
for \Mnat-convex functions. 
Part (4) for s.e.~jump systems is due to \cite{Mmnatjump19},
and Part (5) for c.p.~jump systems is immediate from a result in \cite[Section 3]{BouC95}.
The statements given in Theorem~\ref{THmsetMinkow}
follow from the corresponding statements for functions 
in Theorem~\ref{THmfnconvol}.
\finbox
\end{remark}

In contrast, other kinds of discrete convexity
are not compatible with the Minkowski sum operation, as follows.

\begin{itemize}
\item
The Minkowski sum of integrally convex sets
is not necessarily integrally convex.
In Example~\ref{EXicdim2sumhole},
both $S_{1}$ and $S_{2}$ are integrally convex,
but $S_{1}+S_{2}$ is not.
See also Example~\ref{EXminkow3lnatset}.

\item
The Minkowski sum of L$^{\natural}$-convex sets
is not necessarily L$^{\natural}$-convex. 
See Example~\ref{EXlnatsetMsum}.

\item
The Minkowski sum of L-convex sets
is not necessarily L-convex. 
See Example~\ref{EXlsetsum}.

\item
The Minkowski sum of multimodular sets
is not necessarily multimodular. 
See Example~\ref{EXmmsetMsum}.

\item
The Minkowski sum of discrete midpoint convex sets
is not necessarily discrete midpoint convex.
In Example~\ref{EXicdim2sumhole},
both $S_{1}$ and $S_{2}$ are discrete midpoint convex,
but $S_{1}+S_{2}$ is not.
\end{itemize}

\begin{example}[{\cite[Example 3.11]{MS01rel}}] \rm \label{EXlnatsetMsum}
The Minkowski sum of
$S_{1} =  \{(0, 0, 0), (1, 1, 0)\}$
and
$S_{2} =$ $ \{(0, 0, 0)$, $(0, 1, 1)\}$
is equal to 
\begin{equation}  \label{lnatsetsum1}
S_{1} + S_{2} = \{(0, 0, 0), (0, 1, 1), (1, 1, 0), (1, 2, 1)\}.
\end{equation}
Both $S_{1}$ and $S_{2}$ are L$^{\natural}$-convex.
For $x=(0, 1, 1)$ and $y=(1, 1, 0)$ in $S_{1} + S_{2}$,
we have
$(x+y)/2 = (1/2, 1, 1/2)$,
for which
$\left\lceil (x+y)/2 \right\rceil = (1, 1, 1) \not\in S_{1} + S_{2}$,
and
$\left\lfloor (x+y)/2 \right\rfloor = (0, 1, 0) \not\in S_{1} + S_{2}$.
Therefore, $S_{1}+S_{2}$ is not L$^{\natural}$-convex.
This example also shows that the Minkowski sum of two
\Lnat-convex sets is not discrete midpoint convex, either.
\finbox
\end{example}

\begin{example}[{\cite[Note 5.11]{Mdcasiam}}] \rm \label{EXlsetsum}
(This is an adaptation of Example~\ref{EXlnatsetMsum} to L-convex sets.)
Let
$S_{1} =  \{(0, 0, 0, 0), (1, 1, 0, 0)\}$,
$S_{2} =$ $ \{(0, 0, 0, 0)$, $(0, 1, 1, 0)\}$,
and
$D = \{ \alpha (1,1,1,1) \mid \alpha \in \ZZ \}$.
Define $\tilde{S}_{1} = S_{1} + D $
and 
$\tilde{S}_{2} = S_{2} + D $, which are both L-convex.
The Minkowski sum 
\[
\tilde{S}_{1} + \tilde{S}_{2} = S_{1} + S_{2} + D
= 
\{(0, 0, 0, 0), (0, 1, 1, 0), (1, 1, 0, 0), (1, 2, 1, 0)\} + D
\]
is not L-convex, since for the elements 
$x=(0, 1, 1, 0)$ and $y=(1, 1, 0, 0)$ of $\tilde{S}_{1} + \tilde{S}_{2}$,
we have
$\left\lceil (x+y)/2 \right\rceil = (1, 1, 1, 0) \not\in \tilde{S}_{1} + \tilde{S}_{2}$
and
$\left\lfloor (x+y)/2 \right\rfloor = (0, 1, 0, 0) \not\in \tilde{S}_{1} + \tilde{S}_{2}$.
\finbox
\end{example}

\begin{example}[\cite{MM19multm}]  \rm \label{EXmmsetMsum}
Let
$S_{1} = \{ (0,0,0), (1,0,-1) \}$ and
$S_{2} = \{ (0,0,0), (0,1,0) \}$.
Both $S_{1}$ and $S_{2}$ are multimodular
($S_{2}$ is in fact an integer box).
However, their Minkowski sum
\[
S_{1} + S_{2} = \{ (0,0,0), (1,0,-1),  (0,1,0), (1,1,-1) \}
\]
is not multimodular.
We can check this directly or via transformation 
to $T_{i} = \{ D\sp{-1} x \mid x \in S_{i} \}$
for $i=1,2$.
We have
$T_{1} = \{ (0,0,0), (1,1,0) \}$ and $T_{2} = \{ (0,0,0), (0,1,1) \}$,
which are \Lnat-convex.
But their Minkowski sum 
\[
T_{1} +  T_{2} = \{(0, 0, 0), (0, 1, 1),  \allowbreak    (1, 1, 0), (1, 2, 1)\}
\]
is not \Lnat-convex,
since for $p=(0, 1, 1)$ and $q=(1, 1, 0)$ in $T_{1} + T_{2}$,
we have
$\left\lceil (p+q)/2 \right\rceil = (1, 1, 1) \not\in T_{1} + T_{2}$
and
$\left\lfloor (p+q)/2 \right\rfloor = (0, 1, 0) \not\in T_{1} + T_{2}$.
Since $T_{1} +  T_{2} =  \{ D\sp{-1} x \mid x \in S_{1} + S_{2} \}$,
this means that $S_{1} +  S_{2}$ is not multimodular.
It is mentioned that this example is based on
Example~\ref{EXlnatsetMsum} for \Lnat-convex sets due to \cite[Example 3.11]{MS01rel}.
\finbox
\end{example}

The Minkowski sum of two L$^{\natural}$-convex sets is integrally convex,
though not L$^{\natural}$-convex.

\begin{proposition} [{\cite[Theorem 8.42]{Mdcasiam}}] \label{PRlsetMinkow}
The Minkowski sum of two L$^{\natural}$-convex sets is integrally convex.
In particular,
the Minkowski sum of two L-convex sets is integrally convex.
\finbox
\end{proposition}

\begin{example} \rm \label{EXlnatsetMsum2}
In Example~\ref{EXlnatsetMsum},
the Minkowski sum $S_{1} + S_{2}$ 
in  \eqref{lnatsetsum1} is not \Lnat-convex, but it is integrally convex.
For $x=(0, 1, 1)$ and $y=(1, 1, 0)$ in $S_{1} + S_{2}$,
the midpoint
$z = (x+y)/2 = (1/2, 1, 1/2)$
can be expressed as 
$z = ((0, 1, 1) + (1, 1, 0))/2$
using two points
$(0, 1, 1)$ and $(1, 1, 0)$ 
in $(S_{1} + S_{2}) \cap N(z)$.
\finbox
\end{example}

The Minkowski sum of three L$^{\natural}$-convex sets
is no longer integrally convex, as the following example shows.

\begin{example}[{\cite[Example 4.12]{MS01rel}}] \rm \label{EXminkow3lnatset}
Let $S_1 = \{(0, 0, 0), (1, 1, 0)\}$, $S_2 = \{(0, 0, 0), (0, 1, 1)\}$, and
$S_3 = \{(0, 0, 0), (1, 0, 1)\}$.
These  three sets are all L$^\natural$-convex.
Their Minkowski sum 
\[
S = S_1 + S_2 + S_3
 = \{(0,0,0),(0,1,1),(1,1,0),(1,0,1),(2,1,1),(1,1,2),(1,2,1),(2,2,2)\}
\]
has a ``hole'' at $(1,1,1)$, i.e.,
$(1,1,1) \in \overline{S}$ and
$(1,1,1) \not\in S$.
Hence $S$ is not integrally convex.
It is worth noting that this gives another example to show
that the Minkowski sum of two integrally convex sets is not necessarily
integrally convex, since $S_{1}$ and $S_{2}+S_{3}$ 
are both integrally convex by
\eqref{setfamL} in Theorem~\ref{THsetclassinclusion} and
Proposition \ref{PRlsetMinkow}.
\finbox
\end{example}

We next consider the Minkowski sum 
of a set $S \subseteq \ZZ\sp{n}$ with an integer box $B$.
If $S$ is equipped with a certain discrete convexity, 
the Minkowski sum $S + B$ also possesses the same kind of
discrete convexity, as follows.

\begin{proposition} \label{PRminkowsetbox}
\quad


\noindent
{\rm (1)} 
The Minkowski sum of an integrally convex set
with an integer box is integrally convex.

\noindent
{\rm (2)} 
The Minkowski sum of an L$^{\natural}$-convex set
with an integer box is L$^{\natural}$-convex.

\noindent
{\rm (3)} 
The Minkowski sum of an L-convex set
with an integer box is L-convex.

\noindent
{\rm (4)} 
The Minkowski sum of an M$^{\natural}$-convex set
with an integer box is M$^{\natural}$-convex.

\noindent
{\rm (5)} 
The Minkowski sum of an s.e.~jump system
with an integer box is an s.e.~jump system.
\finbox
\end{proposition}

\begin{remark} \rm \label{RMsetMsumboxbib}
Here is a supplement to Proposition \ref{PRminkowsetbox}
about the Minkowski sum with an integer box.
Part (1) for integrally convex sets is established in \cite[Theorem 4.1]{MM19projcnvl}.
Part (2) for \Lnat-convex sets is a special case of \cite[Theorem 7.11]{Mdcasiam}.
Part (3) for L-convex sets is a special case of \cite[Theorem 7.10]{Mdcasiam}.
Part (4) for \Mnat-convex sets is a special case of Theorem~\ref{THmsetMinkow} (2),
since an integer box is \Mnat-convex.
Part (5) for s.e.~jump systems is a special case of Theorem~\ref{THmsetMinkow} (4),
since an integer box is an s.e.~jump system.
\finbox
\end{remark}

In this connection we note the following facts.

\begin{itemize}
\item
The Minkowski sum of an M-convex set with an integer box is 
not necessarily M-convex.
See Example~\ref{EXmsetboxsum}.

\item
The Minkowski sum of a constant-parity jump system
with an integer box is 
not necessarily a constant-parity jump system.
See Example~\ref{EXmsetboxsum}.

\item
The Minkowski sum of a discrete midpoint convex set
with an integer box is 
not necessarily discrete midpoint convex.
See Example~\ref{EXdicdim3set}.

\item
The Minkowski sum of a multimodular set
with an integer box is 
not necessarily multimodular.
See Example~\ref{EXmmsetMsum}.
\end{itemize}

\begin{example} \rm \label{EXmsetboxsum}
The Minkowski sum of
$S = \{ (1,0), (0,1) \}$
and
$B = \{ (0,0), (1,0) \}$
is equal to 
\[
S+B = \{ (1,0), (0,1), (2,0), (1,1)  \}.
\]
The set $S$ is an M-convex set (and hence a constant-parity jump system)
and $B$ is an integer box.
The component-sums of the members of $S+B$ are 1 and 2,
and therefore, $S+B$ is neither an M-convex set nor a constant-parity jump system.
\finbox
\end{example}

\begin{example}[{\cite[Example 4.3]{MM19projcnvl}}]  \rm \label{EXdicdim3set}
The Minkowski sum of
$S = \{ (0,0,1), (1,1,0) \}$
and
$B = \{ (0,0,0), (1,0,0) \}$
is equal to 
\[
S+B = \{ (0,0,1), (1,1,0),  (1,0,1), (2,1,0) \}.
\]
The set $S$ is discrete midpoint convex
and $B$ is an integer box.
For
$x = (0,0,1)$ and $y = (2,1,0)$ in $S+B$ we have
$\| x - y \|_{\infty} = 2$,
$\left\lceil \frac{x+y}{2} \right\rceil 
= (1,1,1) \not\in S+B$,
and $\left\lfloor \frac{x+y}{2} \right\rfloor
= (1,0,0) \not\in S+B$.
Therefore, $S+B$ is not discrete midpoint convex.
\finbox
\end{example}

\begin{remark} \rm \label{RMprojFromMinkow}
The Minkowski sum operation contains the projection operation as a special case.
Let
$T =  \{ y \in \ZZ\sp{U} \mid  (y,z) \in S \mbox{ for some $z \in \ZZ\sp{N \setminus U}$} \}$
be the projection of $S$ to $U$
as defined in \eqref{projsetdef}.
Consider 
$B = \{ (y',z') \in \ZZ\sp{U} \times \ZZ\sp{N \setminus U} \mid y' = \veczero_{U} \}$,
which is an integer box.
The Minkowski sum $S + B$ is given as
\begin{align*}
 S + B &= \{ (y,z) + (y',z')  \mid (y,z) \in S,  \  y' = \veczero \}
\\ & =
 \{ (y,z+z') \mid (y,z) \in S, \  z' \in \ZZ\sp{N \setminus U} \}
\\ & =
 \{ (y,z'') \mid  y \in T, \  z'' \in \ZZ\sp{N \setminus U} \} 
\\ & =
 T \times \ZZ\sp{N \setminus U}.
\end{align*}
This shows that the restriction of $S + B$ to $U$
coincides with the projection $T$.
\finbox
\end{remark}





\section{Operations on Discrete Convex Functions}
\label{SCfnope}

\subsection{Operations via change of variables}
\label{SCchangevar}

In this section we consider operations on discrete convex functions 
defined by changes of variables such as origin shift,
coordinate inversion, permutation of variables, and scaling of variables.

Let $f$ be a function on $\ZZ\sp{n}$, i.e.,  
$f: \ZZ\sp{n} \to \RR \cup \{ +\infty \}$.
For an integer vector $b \in \ZZ\sp{n}$, the
{\em origin shift}
of $f$ by $b$ 
means a function $g$ on $\ZZ\sp{n}$ defined by%
\begin{equation} \label{shiftfndef}
 g(y) = f(y-b).
\end{equation}

\begin{proposition} \label{PRfnshift}
The origin shift operation \eqref{shiftfndef} for a function preserves
separable convexity,
integral convexity,
L$^{\natural}$-convexity, L-convexity,
M$^{\natural}$-convexity, M-convexity,
multimodularity, 
global and local discrete midpoint convexity,
jump \Mnat-convexity, 
and jump M-convexity.
\finbox
\end{proposition}

For coordinate inversions, two types are distinguished.
The {\em simultaneous coordinate inversion} for $f$
means a function $g$ on $\ZZ\sp{n}$ defined by
\begin{equation} \label{signinvfndef}
g(y_{1},y_{2}, \ldots, y_{n}) = f(-y_{1},-y_{2}, \ldots, -y_{n}),
\end{equation}
and the {\em independent coordinate inversion} for $f$ 
means a function $g$ of $\ZZ\sp{n}$ defined by
\begin{equation} \label{indepsigninvfndef}
g(y_{1},y_{2}, \ldots, y_{n}) = 
f(\tau_{1} y_{1}, \tau_{2} y_{2}, \ldots,  \tau_{n}y_{n})
\end{equation}
with an arbitrary choice of $\tau_{i} \in \{ +1, -1 \}$ $(i=1,2,\ldots,n)$.

\begin{proposition} \label{PRfnsigninv}
The simultaneous coordinate inversion operation \eqref{signinvfndef} 
for a function preserves
separable convexity,
integral convexity,
L$^{\natural}$-convexity, L-convexity,
M$^{\natural}$-convexity, M-convexity,
multimodularity, global and local discrete midpoint convexity,
jump \Mnat-convexity, 
and jump M-convexity.
\finbox
\end{proposition}

\begin{proposition} \label{PRfnindepsigninv}
The independent coordinate inversion operation \eqref{indepsigninvfndef} 
for a function preserves
separable convexity, integral convexity, 
jump \Mnat-convexity, 
and jump M-convexity.
\finbox
\end{proposition}

The independent coordinate inversion operation 
does not preserve
L$^{\natural}$-convexity, L-convexity,
M$^{\natural}$-convexity, M-convexity,
multimodularity, or
(global and local) discrete midpoint convexity.

\begin{example} \rm \label{EXlfnsigninv} 
Let $f(x)= | x_{1} - x_{2} -3 |$
and $g(x)= | x_{1} + x_{2} -3 |$ for $x=(x_{1},x_{2}) \in \ZZ\sp{2}$,
where $g$ is obtained from $f$ by an independent coordinate inversion 
\eqref{indepsigninvfndef} with $\tau_{1}=+1$ and $\tau_{2}=-1$.
The function $f$ is L-convex,  and hence  \Lnat -convex
and globally discrete midpoint convex,
whereas $g$ is not L-convex,  \Lnat -convex, or globally discrete midpoint convex.
Indeed, for $x=(3,0)$ and $y=(0,3)$ we have
$\| x - y \|_{\infty} = 3$, 
$\left\lceil \frac{x+y}{2} \right\rceil = (2,2)$, and
$\left\lfloor \frac{x+y}{2} \right\rfloor = (1,1)$,
for which  
$ g(x) + g(y) = 0 + 0 = 0$ 
is strictly smaller than
$ g \left(\left\lceil \frac{x+y}{2} \right\rceil\right) 
  + g \left(\left\lfloor \frac{x+y}{2} \right\rfloor\right) 
= 1 + 1 = 2$.
\finbox
\end{example}

\begin{example} \rm \label{EXlfnsigninv3} 
Let $f(x)= | x_{1} - x_{2} -1 |$
and $g(x)= | x_{1} + x_{2} -1 |$ for $x=(x_{1},x_{2},x_{3}) \in \ZZ\sp{3}$,
where $g$ is obtained from $f$ by an independent coordinate inversion 
\eqref{indepsigninvfndef} with $\tau_{1}=+1$, $\tau_{2}=-1$, and $\tau_{3}=+1$.
The function $f$ is locally discrete midpoint convex
(actually  L-convex),
whereas $g$ is not locally discrete midpoint convex.
Indeed, for $x=(0,1,2)$ and $y=(1,0,0)$ we have
$\| x - y \|_{\infty} = 2$, 
$\left\lceil \frac{x+y}{2} \right\rceil = (1,1,1)$, and
$\left\lfloor \frac{x+y}{2} \right\rfloor = (0,0,1)$,
for which  
$ g(x) + g(y) = 0 + 0 = 0$ 
is strictly smaller than
$ g \left(\left\lceil \frac{x+y}{2} \right\rceil\right) 
  + g \left(\left\lfloor \frac{x+y}{2} \right\rfloor\right) 
= 1 + 1 = 2$.
\finbox
\end{example}

\begin{example} \rm \label{EXmfnsigninv} 
Let $f(x)= | x_{1} + x_{2} |$ 
and $g(x)= | x_{1} - x_{2} |$ for $x=(x_{1},x_{2}) \in \ZZ\sp{2}$,
where $g$ is obtained from $f$ by an independent coordinate inversion 
\eqref{indepsigninvfndef} with $\tau_{1}=+1$ and $\tau_{2}=-1$.
The function $f$ is M-convex  and hence  \Mnat -convex,
whereas $g$ is not M-convex or \Mnat -convex.
The function $f$ is multimodular, and $g$ is not.
\finbox
\end{example}

For a permutation $\sigma$ of $(1,2,\ldots,n)$,
the {\em permutation} of $f$ by $\sigma$  
means a function $g$ on $\ZZ\sp{n}$ defined by
\begin{equation} \label{permfndef}
 g(y_{1},y_{2}, \ldots, y_{n}) =  f(y_{\sigma(1)}, y_{\sigma(2)}, \ldots, y_{\sigma(n)}) .
\end{equation}

\begin{proposition} \label{PRfnperm}
The permutation operation \eqref{permfndef} for a function preserves
separable convexity,
integral convexity,
L$^{\natural}$-convexity, L-convexity,
M$^{\natural}$-convexity, M-convexity,
global and local discrete midpoint convexity,
jump \Mnat-convexity, 
and jump M-convexity.
\finbox
\end{proposition}

The permutation operation does not preserve multimodularity.
We show two examples, an indicator function
and a quadratic function on the entire lattice $\ZZ\sp{3}$.

\begin{example}  \rm \label{EXmmfnperm3} 
Let $f$ be the indicator function 
$\delta_{S}$ of the multimodular set
$S = \{ (0,0,0),  \allowbreak  (0,1,-1),  \allowbreak  (0,1,0), (1,0,0) \}$
considered in Example~\ref{EXmmsetproj}.
The permutation $g$ of $f$ by
$\sigma: (1,2,3) \mapsto (2,1,3)$ is 
the indicator function of the set 
$T = \{ (0,0,0), (1,0,-1), (1,0,0), (0,1,0) \}$.
Since $T$ is not multimodular,
as shown in Example~\ref{EXmmsetproj},
the function $g$ is not multimodular.
\finbox
\end{example}

\begin{example}[{\cite[Example 3.1]{MM19multm}}] \rm \label{EXmmfnperm1}
The quadratic function
$f(x) = x^{\top} A x$
with
$A  = {\small\footnotesize
\left[ \begin{array}{rrr}
 1 & 1 & 0 \\
 1 & 2 & 1 \\
 0 & 1 & 1
\end{array}\right]}$
is multimodular,
since $A$ satisfies the condition (\ref{mmfquadrcond}).
On the other hand,
$g(y_{1}, y_{2}, y_{3}) = f(y_{2}, y_{1}, y_{3})$
resulting from a transposition is not multimodular.
Indeed we have
$g(y) = y^{\top} \tilde A y$
with
$\tilde A  = {\small\footnotesize
\left[ \begin{array}{rrr}
 2 & 1 & 1 \\
 1 & 1 & 0 \\
 1 & 0 & 1  \\
\end{array}\right]}$,
for which the condition (\ref{mmfquadrcond}) fails for $(i,j)=(1,3)$.
A cyclic permutation of variables results in 
$f(y_{3}, y_{1}, y_{2})$, which 
is not multimodular, 
since it coincides with $y^{\top} \tilde A y$. 
\finbox
\end{example}

It is known that reversing the ordering of variables preserves multimodularity.

\begin{proposition} [{\cite[Remarks (1)]{Haj85}}] \label{PRmmfnvarinv}
For a multimodular function $f$,
the function $g$ defined by
$g(y_{1}, y_{2}, \ldots, y_{n}) \allowbreak = f(y_{n}, \ldots, y_{2}, y_{1})$
is multimodular.
\finbox
\end{proposition}

For a positive integer $\alpha$,
the {\em variable-scaling} 
(or {\em scaling} for short) of $f$ by $\alpha$
means a function $g$ on $\ZZ\sp{n}$ defined by
\begin{equation} \label{scalefndef}
g(y_{1},y_{2}, \ldots, y_{n}) = f(\alpha y_{1}, \alpha y_{2}, \ldots, \alpha y_{n}).
\end{equation}
Note that the same scaling factor $\alpha$ is used for all coordinates.

L-convexity and its relatives are well-behaved 
with respect to the scaling operation.%
\footnote{
The scaled function $g$ may have an empty effective domain.
Therefore, strictly speaking, we should 
add a proviso in Proposition \ref{PRfnscale}
that the resulting function has a nonempty effective domain.
}

\begin{proposition} \label{PRfnscale}
The variable-scaling operation \eqref{scalefndef} for a function preserves
separable convexity,
L$^{\natural}$-convexity, L-convexity,
multimodularity,  and discrete midpoint convexity.
\finbox
\end{proposition}

\begin{remark} \rm \label{RMfnscalebib}
Here is a supplement to Proposition \ref{PRfnscale} about scaling.
The scaling operation for L-convex (resp., \Lnat-convex) functions is treated in 
\cite[Theorem 7.10]{Mdcasiam} 
(resp., \cite[Theorem 7.11]{Mdcasiam}).
The result for multimodular functions 
is due to \cite[Proposition 7]{MM19multm},
and that for 
discrete midpoint convex functions is to \cite[Theorem 9]{MMTT19dmpc}.
\finbox
\end{remark}

In contrast, M-convexity and integral convexity
are not compatible with the scaling operation, as follows.
In referring to examples in Section \ref{SCsetope},
we intend to consider the indicator functions
of the sets mentioned in the examples.

\begin{itemize}

\item
The scaling of an \Mnat-convex function
is not necessarily \Mnat-convex
(Example~\ref{EXmnatsetscdim3}).

\item
The scaling of an M-convex function
is not necessarily M-convex
(Example~\ref{EXmsetscdim3}).

\item
The scaling of a jump \Mnat-convex (resp., M-convex) function
is not necessarily jump \Mnat-convex (resp., M-convex)
(Examples \ref{EXmnatsetscdim3} and \ref{EXmsetscdim3}).

\item
The scaling of an integrally convex function
is not necessarily integrally convex.
See  Examples \ref{EXscICfnNG422indic} and \ref{EXscICfnNG422} below.
\end{itemize}

For the scaling of an integrally convex function,
we show two examples, an indicator function
and a function defined on an integer box in $\ZZ\sp{3}$.

\begin{example} \rm \label{EXscICfnNG422indic}
Let $f$ be the indicator function 
$\delta_{S}$ of the integrally convex set $S$
considered in Example~\ref{EXscICsetNG422} (Fig.~\ref{FGicsetsc}).
The scaled function $g(y) = f(2y)$ 
is the indicator function of the set 
$T = \{ (0,0,0), (1,0,0),  \allowbreak  (1,0,1), (2,1,1) \}$.
Since $T$ is not integrally convex,
as shown in Example~\ref{EXscICsetNG422},
the function $g$ is not integrally convex.
\finbox
\end{example}

\begin{example}[{\cite[Example 3.1]{MMTT19proxIC}}] \rm \label{EXscICfnNG422}
Consider the integrally convex function
$f: \ZZ^{3} \to \RR \cup \{ +\infty \}$
defined on an integer box $[(0,0,0), (4,2,2)]_{\ZZ}$ by
\[
\begin{array}{c|rrrrrl}
  \multicolumn{1}{c}{x_{2}}  &  \multicolumn{5}{c}{f(x_{1},x_{2},0)}  
\\ \cline{2-6}
 2  & 3 & 1 & 1 & 1 & \multicolumn{1}{r|}{3}
\\
 1  & 1 & 0 & 0 & 0 & \multicolumn{1}{r|}{0} 
\\
 0  & 0 & 0 & 0 & 0 & \multicolumn{1}{r|}{3}
\\ \cline{1-6}
  & 0 & 1 & 2 & 3 & 4 & x_{1} 
\end{array}  
\quad
\begin{array}{c|rrrrrl}
  \multicolumn{1}{c}{x_{2}}  &  \multicolumn{5}{c}{f(x_{1},x_{2},1)}  
\\ \cline{2-6}
 2  & 2 & 1 & 0 & 0 & \multicolumn{1}{r|}{0}
\\
 1  & 1 & 0 & 0 & 0 & \multicolumn{1}{r|}{0} 
\\
 0  & 0 & 0 & 0 & 0 & \multicolumn{1}{r|}{0}
\\ \cline{1-6}
  & 0 & 1 & 2 & 3 & 4 & x_{1} 
\end{array}  
\quad
\begin{array}{c|rrrrrl}
  \multicolumn{1}{c}{x_{2}}  &  \multicolumn{5}{c}{f(x_{1},x_{2},2)}  
\\ \cline{2-6}
 2  & 3 & 2 & 1 & 0 & \multicolumn{1}{r|}{0}
\\
 1  & 2 & 1 & 0 & 0 & \multicolumn{1}{r|}{0} 
\\
 0  & 3 & 0 & 0 & 0 & \multicolumn{1}{r|}{3}
\\ \cline{1-6}
  & 0 & 1 & 2 & 3 & 4 & x_{1} 
\end{array}  
\]
For the scaling with $\alpha = 2$, 
the function $g(y) = f(2y)$ is not integrally convex.
Indeed, 
the inequality (\ref{intcnvconddist2}) in Theorem \ref{THfavtarProp33}
fails for $g$
with $y = (0,0,0)$ and $z = (2, 1, 1)$ as follows:
\begin{align*}
\tilde {g} \, \bigg(\frac{y + z}{2} \bigg)
&= 
\tilde {g} \, \bigg( 1,\frac{1}{2}, \frac{1}{2} \bigg)
\\
&= \min \{
\frac{1}{2}g(1,1,1) + \frac{1}{2}g(1,0,0), \ 
\frac{1}{2}g(1,1,0) + \frac{1}{2}g(1,0,1) \}
\\
 & = \frac{1}{2}  \min \{ f(2,2,2) + f(2,0,0), \  f(2,2,0) + f(2,0,2) \}  
\\
 & = \frac{1}{2}  \min \{ 1+ 0, \   1 + 0  \} = \frac{1}{2} ,
\\ 
\frac{1}{2} (g(y) + g(z))
&= 
\frac{1}{2} ( f(0,0,0) + f(4,2,2) ) = 0.
\end{align*}
It is noted that  
the minimizer set $\{ x \mid f(x)=0 \}$
coincides with the integrally convex set $S$
considered in Examples \ref{EXscICsetNG422} and \ref{EXscICfnNG422indic}.
\finbox
\end{example}

In the case of $n = 2$,
M-convexity and integral convexity admit the scaling operation.

\begin{proposition} \label{PRintcnvfnscdim2}
Let $\alpha$ be a positive integer,
$f: \ZZ^{2} \to \RR \cup \{ +\infty  \}$
a function in two variables, and 
$g(y) = f(\alpha y)$ for $y \in \ZZ\sp{2}$.

\noindent
{\rm (1)}
If $f$ is integrally convex, then $g$ is integrally convex.


\noindent
{\rm (2)}
If $f$ is \Mnat-convex, then $g$ is \Mnat-convex.

\noindent
{\rm (3)}
If $f$ is M-convex, then $g$ is M-convex.
\finbox
\end{proposition}

\begin{remark} \rm \label{RMfndim2scalebib}
Here is a supplement to Proposition \ref{PRintcnvfnscdim2} about functions in two variables.
Part (1) for integrally convex functions is due to 
\cite[Theorem 3.2]{MMTT19proxIC}.
Part (2) for \Mnat-convex functions follows from 
the statement for \Lnat-convex functions in
Proposition \ref{PRfnscale} and Remark \ref{RMfndim2} in Section~\ref{SCclassfnZ}.
Part (3) for M-convex functions is almost a triviality, since
an M-convex function in two variables is essentially a univariate
convex function in the sense of \eqref{univarconvdef},
which is a special case of separable-convex functions
treated in Proposition \ref{PRfnscale}.
\finbox
\end{remark}

Table \ref{TBoperation1dcfnZ}
is a summary of the behavior of discrete convex functions 
with respect to changes of variables
such as origin shift, 
coordinate inversion, permutation, and scaling.
In this table,
``Y'' means ``Yes, this function class is closed under this operation''
and  ``\textbf{\textit{N}}'' means ``No, this function class is not closed under this operation.''

\begin{table}
\begin{center}
\caption{Operations on discrete convex functions via coordinate changes}
\label{TBoperation1dcfnZ}

\medskip


\begin{tabular}{l|c|cc|cc|l}
 Discrete    & Origin & \multicolumn{2}{c|}{Coord.~inversion}  & Permu- & Variable-
 & References
\\ \cline{3-4}
 \quad convexity & shift &  simult. & indep. & tation   & scaling
\\ 
  & $f(x-b)$ & $f(-x)$ & $f(\pm x_{i})$ & $f(x_{\sigma(i)})$  & $f(\alpha x)$  
\\ \hline
 Separable convex  & \YES & \YES & \YES & \YES & \YES  
\\ 
 Integrally convex & \YES & \YES & \YES  & \YES & \NO   
   & \cite{FT90,MMTT19proxIC}
\\ 
\Lnat-convex       & \YES & \YES & \NO & \YES & \YES  
  & \cite{FM00,Mdca98} 
\\ 
L-convex           & \YES & \YES & \NO & \YES & \YES
  & \cite{Mdca98}
\\ 
\Mnat-convex       & \YES & \YES & \NO& \YES & \NO 
  & \cite{Mstein96,MS99gp}
 \\ 
M-convex           & \YES & \YES & \NO & \YES & \NO 
  & \cite{Mstein96}
\\ 
Multimodular       & \YES & \YES & \NO  &  \NO  & \YES 
  & \cite{AGH00,AGH03,GY94mono,Haj85,MM19multm}
\\ 
Globally d.m.c.       & \YES & \YES & \NO & \YES & \YES 
  & \cite{MMTT19dmpc}
\\ 
Locally  d.m.c.       & \YES & \YES & \NO & \YES & \YES 
  & \cite{MMTT19dmpc}
\\ 
Jump \Mnat-convex   & \YES & \YES & \YES & \YES & \NO 
  &  \cite{Mmjump06,Mmnatjump19} 
\\ 
Jump M-convex   & \YES & \YES & \YES & \YES & \NO 
  &  \cite{KMT07jump,Mmjump06} 
\\ \hline
\multicolumn{3}{l}{d.m.c.: discrete midpoint convex.} 
\\
\multicolumn{7}{l}{``Y'' means ``Yes, this function class is closed under this operation.''} 
\\
\multicolumn{7}{l}{``\textbf{\textit{N}}'' means ``No, this function class is not closed under this operation.''} 
\end{tabular}
\end{center}
\end{table}

In Sections \ref{SCvalscalefn}--\ref{SCconvol} 
we consider operations such as 
nonnegative multiplication of function values,
addition of a linear function,
projection (partial minimization),
sum of two functions, and
convolution of two functions.
Table~\ref{TBoperation2dcfnZ} 
is a summary of the behavior of discrete convex functions 
with respect to those operations.
Two types of additions are distinguished:
one is the addition of $f$ with a separable convex function $\varphi$,
denoted $f + \varphi$,
and the other is the sum
$f_{1} + f_{2}$
of two functions $f_{1}$ and $f_{2}$ 
in the same convexity class. 
Similarly for convolution,
denoted $f \conv \varphi$ and $f_{1} \conv f_{2}$.
In Section \ref{SCfnconj} we consider conjugate and biconjugate functions.

\begin{table}
\begin{center}
\caption{Operations on discrete convex functions related to function values}
\label{TBoperation2dcfnZ}

\medskip


\addtolength{\tabcolsep}{-4pt}
\begin{tabular}{l|ccc|cc|cc|l}
 Discrete   & Value- \  & Restric- & Projec- & \multicolumn{2}{c|}{Addition} & \multicolumn{2}{c|}{Convolution} 
 & Reference
 \\ \cline{5-8}
 \quad convexity & scaling & tion   & tion  &  separ. & general &   separ. & general 
\\ 
  & $a f $ & $f(y,\bm{0})$  & $\displaystyle \inf_{z} f(y,z)$   
  & $f + \varphi$ & $f_{1} + f_{2}$ & $f \conv \varphi$ & $f_{1} \conv f_{2}$
\\ \hline
 Separable convex & \YES  & \YES  & \YES  & \YES  & \YES  & \YES  & \YES 
\\ 
 Integrally convex & \YES  & \YES  & \YES  & \YES  
  & \NO   & \YES  & \NO
   & \cite{FT90,MM19projcnvl,MS01rel}
\\ 
\Lnat-convex  & \YES  & \YES  & \YES  & \YES  & \YES  & \YES  & \NO 
  & \cite{FM00,Mdca98} 
\\ 
L-convex  & \YES  & \NO  & \YES  
  & \NO & \YES  & \YES  
  & \NO 
  & \cite{Mdca98}
\\ 
\Mnat-convex  & \YES  &  \YES  & \YES  & \YES  & \NO 
  & \YES  & \YES 
  & \cite{Mstein96,MS99gp}
\\ 
M-convex & \YES  & \YES  & \NO & \YES  
  & \NO & \NO & \YES 
  & \cite{Mstein96}
\\ 
Multimodular  & \YES  &  \YES   & \NO  
   & \YES  & \YES  & \NO & \NO 
  & \cite{AGH00,Haj85,MM19multm}
\\ 
Globally d.m.c.  & \YES   &  \YES   & \YES  & \YES  & \YES  
  & \NO  & \NO 
  & \cite{MMTT19dmpc}
\\ 
Locally d.m.c.  & \YES  &  \YES   &  \YES  & \YES  & \YES  
  & \NO & \NO 
  & \cite{MMTT19dmpc}
\\ 
Jump \Mnat-convex   & \YES  & \YES   & \YES  & \YES  
  & \NO & \YES & \YES 
  &  \cite{Mmjump06,Mmnatjump19} 
\\ 
Jump M-convex   & \YES  & \YES   & \NO & \YES  & \NO 
  & \NO  & \YES 
  &  \cite{KMT07jump,Mmjump06} 
\\ \hline
\multicolumn{1}{l}{} &
\multicolumn{7}{r}{($\varphi$: separable convex)} 
\\
\multicolumn{8}{l}{``Y'' means ``Yes, this function class is closed under this operation.''} 
\\
\multicolumn{8}{l}{``\textbf{\textit{N}}'' means ``No, this function class is not closed under this operation.''} 
\end{tabular}
\addtolength{\tabcolsep}{4pt}
\end{center}
\end{table}

\subsection{Value-scaling}
\label{SCvalscalefn}

For a function
$f: \ZZ\sp{n} \to \RR \cup \{ +\infty \}$ 
and a nonnegative factor $a \geq 0$,
the {\em value-scaling} of $f$ by $a$ 
means  a function $g: \ZZ\sp{n} \to \RR \cup \{ +\infty \}$ defined by%
\begin{equation} \label{valscalefndef}
 g(y) = a f(y)
  \qquad (y \in \ZZ\sp{n}) .
\end{equation}

\begin{proposition} \label{PRfnvalscale}
The value-scaling operation \eqref{valscalefndef} for a function preserves
separable convexity,
integral convexity,
L$^{\natural}$-convexity, L-convexity,
M$^{\natural}$-convexity, M-convexity,
multimodularity, 
global and local discrete midpoint convexity,
jump \Mnat-convexity,
and jump M-convexity.
\finbox
\end{proposition}

\subsection{Restriction}
\label{SCrestrfn}

Let $N = \{ 1,2,\ldots, n \}$.
For a function
$f: \ZZ\sp{N} \to \RR \cup \{ +\infty \}$ 
and a subset $U$ of the index set $N$,
the
{\em restriction}
of $f$ to $U$ is a function $g: \ZZ\sp{U} \to \RR \cup \{ +\infty \}$ defined by%
\footnote{
For any $z\in \ZZ\sp{N \setminus U}$ we may consider a function 
$g(y) = f(y,z)$ in $y \in \ZZ\sp{U}$
as a restriction of $f$ to $U$.
For simplicity of description we choose $z=\veczero_{N \setminus U}$.
} 
\begin{equation} \label{restrfndef}
 g(y) = f(y,\veczero_{N \setminus U})
  \qquad (y \in \ZZ\sp{U}) ,
\end{equation}
where $\veczero_{N \setminus U}$ denotes the zero vector 
in $\ZZ\sp{N \setminus U}$.
The notation 
$(y,\veczero_{N \setminus U})$ means the vector
whose $i$th component is equal to $y_{i}$ for $i \in U$
and to 0 for $i \in N \setminus U$;
for example, if $N = \{ 1,2,3 \}$ and $U = \{ 1,3 \}$,
$(y,\veczero_{N \setminus U})$ means $(y_{1}, 0, y_{3})$.

The restriction operation preserves discrete convexity as follows.%
\footnote{
The restriction operation may result in a function with an empty effective domain.
Therefore, strictly speaking, we should 
add a proviso in Proposition \ref{PRfnrestr}
that the resulting function has a nonempty effective domain.
}

\begin{proposition} \label{PRfnrestr}
The restriction operation \eqref{restrfndef} for a function preserves
separable convexity,
integral convexity,
L$^{\natural}$-convexity, 
M$^{\natural}$-convexity, M-convexity,
multimodularity, 
global and local discrete midpoint convexity,
jump \Mnat-convexity,
and jump M-convexity.
\finbox
\end{proposition}

\begin{remark} \rm \label{RMfnrestrbib}
Here is a supplement to Proposition \ref{PRfnrestr} about restriction.
The restriction of a separable convex function
$\sum \{ \varphi_{i}(x_{i}) \mid i \in N \}$ 
to $U$ is given by the separable convex function
$\sum \{ \varphi_{i}(x_{i}) \mid i \in U \}  + C$,
where 
$C = \sum \{  \varphi_{j}(0) \mid j \in N \setminus U \}$.
The result for multimodular functions 
is shown in \cite[Lemma 2.3]{AGH00} as well as in \cite[Lemma 3]{AGH03}.
The statements for other kinds of discrete convex functions are rather obvious.
We can find the statement
for L$^{\natural}$-convexity in \cite[Theorem 7.11]{Mdcasiam},
for M$^{\natural}$-convexity in \cite[Theorem 6.14]{Mdcasiam},
for M-convexity in \cite[Theorem 6.13]{Mdcasiam},
for integral convexity in \cite[Proposition 3.19]{Mdcasiam} 
(in a more general form as in Remark \ref{RMfnrestrbox}),
for global and local discrete midpoint convexity in \cite[Table 5.1]{MM19multm}, and
for jump \Mnat- and M-convexity in \cite{Mmnatjump19} and \cite[Section 3]{KMT07jump}.
\finbox
\end{remark}

The restriction of an L-convex function is not necessarily L-convex,
simply because the restricted function lacks in
the linearity (\ref{shiftlfnZ}) in the direction of $\vecone$.

\begin{example}[cf., Example~\ref{EXlsetrestr}]  \rm \label{EXlfnrestr} 
For $S = \{ (t,t) \mid t \in \ZZ \}$, its indicator function $\delta_{S}$ 
is L-convex.
The restriction of $f=\delta_{S}$ to $U = \{ 1 \}$ is given by 
$g=\delta_{T}$ for $T = \{ 0 \}$,
which is not L-convex.
\finbox
\end{example}

\begin{remark} \rm \label{RMfnrestrbox}
The restriction of a function $f: \ZZ\sp{N} \to \RR \cup \{ +\infty \}$
to an integer box
$[a,b]_{\ZZ}$
is a function $f_{[a,b]}$ on $\ZZ\sp{N}$ defined by 
\begin{equation} \label{fboxrestrict}
 f_{[a,b]}(x) = 
   \left\{  \begin{array}{ll}
     f(x)      & (x \in [a,b]_{\ZZ}), \\
     +\infty   & (x \not\in [a,b]_{\ZZ}) .
                      \end{array}  \right.
\end{equation}
In the special case where 
$a_{i}=-\infty$ and $b_{i}=+\infty$ for $i \in U$, and
$a_{i}=b_{i}=0$ for $i \in N \setminus U$,
this function $f_{[a,b]}$ coincides with 
$f(y,\veczero_{N \setminus U})$ in \eqref{restrfndef}.
Thus, the restriction to an integer box is more general
than the restriction to a subset of $N$.
\finbox
\end{remark}

\subsection{Projection}
\label{SCproj}

For a function
$f: \ZZ\sp{N} \to \RR \cup \{ +\infty \}$ 
and a subset $U$ of the index set $N = \{ 1,2,\ldots, n \}$,
the {\em projection}
of $f$ to $U$ is a function
$g: \ZZ\sp{U} \to \RR \cup \{ -\infty, +\infty \}$
defined by
\begin{equation} \label{projfndef} 
  g(y)  =  \inf \{ f(y,z) \mid z \in \ZZ\sp{N \setminus U} \}
  \qquad (y \in \ZZ\sp{U}) ,
\end{equation}
where the notation $(y,z)$ means the vector
whose $i$th component is equal to $y_{i}$ for $i \in U$
and to $z_{i}$ for $i \in N \setminus U$.
We assume $g > -\infty$.
The projection is sometimes called {\em partial minimization}.

The projection operation preserves discrete convexity as follows.

\begin{proposition} \label{PRfnproj}
The projection operation \eqref{projfndef} for a function preserves
separable convexity,
integral convexity,
L$^{\natural}$-convexity, L-convexity,
M$^{\natural}$-convexity, 
global and local discrete midpoint convexity,
and jump \Mnat-convexity.
\finbox
\end{proposition}

\begin{remark} \rm \label{RMfnprojbib}
Here is a supplement to Proposition \ref{PRfnproj} about projection.
The projection of a separable convex function
$\sum \{ \varphi_{i}(x_{i}) \mid i \in N \}$ 
to $U$ is given by the separable convex function
$\sum \{ \varphi_{i}(x_{i}) \mid i \in U \} + C$,
where 
$C = \sum \{ \min \varphi_{j} \mid j \in N \setminus U \}$.
The results for 
 L$^{\natural}$-convex, L-convex, and 
M$^{\natural}$-convex functions
are given in 
\cite[Theorem 7.11]{Mdcasiam}, 
\cite[Theorem 7.10]{Mdcasiam}, and
\cite[Theorem 6.15]{Mdcasiam},
respectively.
The statements for integrally convex functions,
 globally discrete midpoint convex functions,
locally discrete midpoint convex functions, and
jump \Mnat-convex functions
are obtained recently in
\cite[Theorem 3.3]{MM19projcnvl},
\cite[Theorem 3.5]{MM19projcnvl}, 
\cite[Theorem 3.6]{MM19projcnvl}, and
\cite{Mmnatjump19}, 
respectively.
\finbox
\end{remark}

The projection of an M-convex function is not necessarily M-convex,
simply because the effective domain of 
the projected function does not lie on a hyperplane of a constant component sum.
Similarly, the projection of a jump M-convex function is not necessarily jump M-convex.

\begin{example}  \rm \label{EXmfnproj} 
For $S = \{ (t,-t) \mid t \in \ZZ \}$, its indicator function $\delta_{S}$ is M-convex.
The projection of $f=\delta_{S}$ to $U = \{ 1 \}$ is given by 
$g \equiv 0$ on $\ZZ$
(i.e., $g(y_{1}) = 0$ for all $y_{1} \in \ZZ$),
which is not M-convex.
The function $f$ is jump M-convex, while $g$ is not.
\finbox
\end{example}

The projection of a multimodular function 
to a subset of indices is not necessarily multimodular.
We show two examples, an indicator function
and a quadratic function on the entire lattice $\ZZ\sp{4}$.

\begin{example}  \rm \label{EXmmfnproj3} 
Let $f$ be the indicator function 
$\delta_{S}$ of the multimodular set
$S = \{ (0,0,0),  \allowbreak  (0,1,-1),  \allowbreak  (0,1,0), (1,0,0) \}$
considered in Example~\ref{EXmmsetproj}.
The projection $g$ of $f$ to $U = \{ 1, 3 \}$ is 
the indicator function of the set 
$T = \{ (0,0), (0,-1), (1,0) \}$.
Since $T$ is not multimodular,
as shown in Example~\ref{EXmmsetproj},
the function $g$ is not multimodular.
\finbox
\end{example}

\begin{example}[{\cite[Example 4.1]{MM19multm}}]   \rm \label{EXmmfnproj4}
The quadratic function
$f(x) = x^{\top} A x$
with
\[
A  = {\footnotesize \small
\left[ \begin{array}{cccc}
  3 & 2 & 1 & 0 \\
 2 & 3 & 2 &  1 \\
 1 & 2 & 2 & 1 \\
 0 & 1 & 1 &  1 \\
\end{array}\right]
}
\]
is multimodular, 
since $A$ satisfies the condition (\ref{mmfquadrcond}).
On the other hand,
the projection $g$ of $f$ to $U= \{ 1,2,4 \}$ is not multimodular.
Indeed we have
$g(y) = y^{\top} \tilde A y$
with
$\tilde A = {\small\footnotesize
{\displaystyle {\small\footnotesize \frac{1}{2}}  }
\left[ \begin{array}{rrr}
 5 & 2 & -1 \\
 2 & 2 & 0 \\
 -1 & 0 & 1 \\
\end{array}\right]}$,
where
$\tilde A =(\tilde a_{ij} \mid  i,j =1,2,4)$
is obtained from $A$ by the usual sweep-out operation:
 $\tilde a_{ij} = a_{ij} - a_{i3} a_{3j} / a_{33}$
$(i,j \in \{ 1,2,4 \})$.
The matrix $\tilde A$  violates the condition (\ref{mmfquadrcond}) for $(i,j)=(1,2)$.
\finbox
\end{example}

The projection of a multimodular function to a consecutive index subset is multimodular.
Recall that a consecutive set 
means a set of the form $\{ k, k+1, \ldots, l -1, l \}$ for some $k \leq l$.

\begin{proposition}[{\cite[Proposition 10]{MM19multm}}] \label{PRmmfnprojinterval}
For a multimodular function $f$ and a consecutive index subset $U$,
the projection of $f$ to $U$ in \eqref{projfndef} is multimodular.
\finbox
\end{proposition}

\subsection{Addition}
\label{SCsumconvfn}

The {\em sum} of two functions 
$f_{1}, f_{2}: \ZZ\sp{n} \to \RR \cup \{ +\infty \}$
is defined, in an obvious way, by
\begin{equation} \label{fnsumdef}
(f_{1} + f_{2})(x)= f_{1}(x) + f_{2}(x)
\qquad (x \in \ZZ\sp{n}) .
\end{equation}
The effective domain of the sum is equal to the intersection 
of the effective domains of the given functions, that is,
\begin{equation} \label{fnsumdom}
 \dom (f_{1} + f_{2}) = \dom f_{1} \cap \dom f_{2}.
\end{equation}
For two sets $S_{1}$, $S_{2} \subseteq \ZZ\sp{n}$, 
the sum of their indicator functions
$\delta_{S_{1}}$, $\delta_{S_{2}}$
coincides with 
the indicator function of their intersection
$S_{1} \cap S_{2}$,
that is,
$\delta_{S_{1}} +  \delta_{S_{2}} = \delta_{S_{1} \cap S_{2}}$.

L-convexity and its relatives are well-behaved with respect to the sum operation.%
\footnote{
By \eqref{fnsumdom} the sum $f_{1} + f_{2}$ may have an empty effective domain.
Therefore, strictly speaking, we should 
add a proviso in Propositions \ref{PRlfnsum}, \ref{PRmfninter}, and \ref{PRfnsepsum}
that the resulting function has a nonempty effective domain.
}

\begin{proposition} \label{PRlfnsum}
\quad

\noindent
{\rm (1)} 
The sum of separable convex functions is separable convex.

\noindent
{\rm (2)} 
The sum of \Lnat-convex functions is \Lnat-convex.

\noindent
{\rm (3)} 
The sum of L-convex functions is L-convex.

\noindent
{\rm (4)} 
The sum of multimodular functions is multimodular.

\noindent
{\rm (5)} 
The sum of globally discrete midpoint convex functions
is globally discrete midpoint convex.

\noindent
{\rm (6)} 
The sum of locally discrete midpoint convex functions
is locally discrete midpoint convex.
\finbox
\end{proposition}

\begin{remark} \rm \label{RMfninterbib}
Here is a supplement to Proposition \ref{PRlfnsum} about sum.
The proofs of these statements (1)--(6) are pretty easy.
The sum of separable convex functions
$\sum_{i=1}\sp{n} \varphi_{1i}(x_{i})$ 
and 
$\sum_{i=1}\sp{n} \varphi_{2i}(x_{i})$ 
is given by 
$\sum_{i=1}\sp{n} (\varphi_{1i} + \varphi_{2i})(x_{i})$,
which is a separable convex function.
Part (2) for \Lnat-convex functions is given in 
\cite[Theorem 7.11]{Mdcasiam}.
Part (3) for L-convex functions
is given in \cite[Theorem 7.10]{Mdcasiam}.
Part (4) for multimodular functions
is given in \cite[Proposition 8]{MM19multm}.
Parts (5) and (6) for globally and locally discrete midpoint convex functions
are given in \cite[Table 5.1]{MM19multm}.
\finbox
\end{remark}

In contrast, M-convexity and integral convexity
are not compatible with the sum operation.
In referring to examples in Section \ref{SCsetope},
we intend to consider the indicator functions
of the sets mentioned in the examples.

\begin{itemize}
\item
The sum of integrally convex functions
is not necessarily integrally convex
(Example~\ref{EXicsetinter}).

\item
The sum of \Mnat-convex functions
is not necessarily \Mnat-convex
(Example~\ref{EXmnatsetinter}).

\item
The sum of M-convex functions
is not necessarily M-convex
(Example~\ref{EXmsetinter}).

\item
The sum of jump \Mnat-convex (resp., M-convex)  functions 
is not necessarily jump \Mnat-convex (resp., M-convex)
(Example~\ref{EXmsetinter}).
\end{itemize}

The sum of two M$^{\natural}$-convex functions is integrally convex,
though not M$^{\natural}$-convex.

\begin{proposition}[{\cite[Theorem 8.31]{Mdcasiam}}]  \label{PRmfninter}
The sum of two M$^{\natural}$-convex functions is integrally convex.
In particular, the sum of two M-convex functions is integrally convex.
\finbox
\end{proposition}

We next consider the sum
of a discrete convex function $f$ with a separable convex function $\varphi$.
If $f$ is equipped with a certain discrete convexity, 
the sum $f + \varphi$ also possesses the same kind of
discrete convexity, as follows.
Note that a linear function
$\sum_{i=1}\sp{n} c_{i} x_{i}$
is a very special case of a separable convex function.

\begin{proposition} \label{PRfnsepsum}
\quad

\noindent
{\rm (1)} 
The sum of an integrally convex function
with a separable convex function is integrally convex.

\noindent
{\rm (2)} 
The sum of an \Lnat-convex function 
with a separable convex function is \Lnat-convex.

\noindent
{\rm (3)} 
The sum of an \Mnat-convex function 
with a separable convex function is \Mnat-convex.

\noindent
{\rm (4)} 
The sum of an M-convex function 
with a separable convex function is M-convex.

\noindent
{\rm (5)} 
The sum of a multimodular function 
with a separable convex function is multimodular.

\noindent
{\rm (6)} 
The sum of a globally discrete midpoint convex function
with a separable convex function is globally discrete midpoint convex.

\noindent
{\rm (7)} 
The sum of a locally discrete midpoint convex function
with a separable convex function is locally discrete midpoint convex.

\noindent
{\rm (8)} 
The sum of a jump \Mnat-convex function 
with a separable convex function is jump \Mnat-convex.

\noindent
{\rm (9)} 
The sum of a jump M-convex function 
with a separable convex function is jump M-convex.
\finbox
\end{proposition}

\begin{remark} \rm \label{RMfnsepsumbib}
Here is a supplement to Proposition \ref{PRfnsepsum}
about the sum with a separable convex function.
The statements (1)--(9) are immediate from the definitions.
Parts (2), (5), (6),  and (7) 
for \Lnat-convex, multimodular, and globally and locally discrete midpoint convex
functions
 are special cases of Proposition~\ref{PRlfnsum} (2), (4), (5), and (6),
respectively, since a separable convex function is 
\Lnat-convex, multimodular, and globally and locally discrete midpoint convex
(cf., Theorem \ref{THfnclassinclusion}).
\finbox
\end{remark}

In this connection we note the following facts.

\begin{itemize}
\item
The sum of an L-convex function $f$ with a separable convex function $\varphi$ is 
not necessarily L-convex.
This is because the sum $f + \varphi$ lacks in
the linearity (\ref{shiftlfnZ}) in the direction of $\vecone$
(Example \ref{EXlsetboxinter}).

\item
The sum of an L-convex function $f$ with a linear function 
$\langle c,x \rangle = \sum_{i=1}\sp{n} c_{i} x_{i}$
is  L-convex.
Let $g(x) = f(x) + \langle c,x \rangle$
and $r_{f} = f(x+ \bm{1}) - f(x)$.
For the linearity (\ref{shiftlfnZ}) in the direction of $\vecone$,
we observe that 
$g(x + \bm{1})
 = f(x+ \bm{1}) + \langle c,x+ \bm{1} \rangle
 = f(x) + r_{f} + \langle c,x \rangle + \langle c,\bm{1} \rangle
 = g(x) + r_{g}
$
with $r_{g} = r_{f} + \langle c,\bm{1} \rangle$.
\end{itemize}

\subsection{Convolution}
\label{SCconvol}

The (infimal) {\em convolution} of two functions
$f_{1}, f_{2}: \ZZ\sp{n} \to \RR \cup \{ +\infty \}$
is defined by
\begin{equation} \label{f1f2convdef}
(f_{1} \conv f_{2})(x) =
 \inf\{ f_{1}(y) + f_{2}(z) \mid x= y + z, \  y, z \in \ZZ\sp{n}  \}
\qquad (x \in \ZZ\sp{n}) ,
\end{equation}
where it is assumed that the infimum 
is bounded from below (i.e., $(f_{1} \conv f_{2})(x) > -\infty$) 
for every $x \in \ZZ\sp{n}$.
For two sets $S_{1}$, $S_{2} \subseteq \ZZ\sp{n}$, 
the convolution of their indicator functions
$\delta_{S_{1}}$, $\delta_{S_{2}}$
coincides with 
the indicator function of their Minkowski sum
$S_{1}+S_{2} = \{ y + z \mid y \in S_{1}, z \in S_{2} \}$,
that is,
$\delta_{S_{1}} \conv \delta_{S_{2}} = \delta_{S_{1}+S_{2}}$.

M-convexity  and its relatives are well-behaved 
with respect to the convolution operation.

\begin{theorem} \label{THmfnconvol}
\quad


\noindent
{\rm (1)} 
The convolution of separable convex functions is a separable convex function.

\noindent
{\rm (2)} 
The convolution of \Mnat-convex functions is \Mnat-convex.

\noindent
{\rm (3)} 
The convolution of M-convex functions is M-convex.

\noindent
{\rm (4)} 
The convolution of jump \Mnat-convex functions is jump \Mnat-convex.

\noindent
{\rm (5)} 
The convolution of jump M-convex functions is jump M-convex.
\finbox
\end{theorem}

\begin{remark} \rm \label{RMmfnconvolbib}
Here is a supplement to Theorem~\ref{THmfnconvol} about convolution.
The convolution of separable convex functions
$\sum_{i=1}\sp{n} \varphi_{1i}(x_{i})$ 
and 
$\sum_{i=1}\sp{n} \varphi_{2i}(x_{i})$ 
is given by 
$\sum_{i=1}\sp{n} (\varphi_{1i} \conv \varphi_{2i})(x_{i})$,
which is a separable convex function.
Part (3) for M-convex functions
originates in \cite[Theorem 6.10]{Mstein96} 
and is given in \cite[Theorem 6.13]{Mdcasiam}.
Part (2) for \Mnat-convex functions is a variant of (3) 
for M-convex functions and is given in 
\cite[Theorem 6.15]{Mdcasiam}.
Part (4) for jump \Mnat-convex functions is given by \cite{Mmnatjump19},
and  Part (5) for jump M-convex functions is by \cite[Theorem 12]{KMT07jump}.
\finbox
\end{remark}

In contrast, other kinds of discrete convexity
are not compatible with the convolution operation, as follows.
In referring to examples in Section \ref{SCsetope},
we intend to consider the indicator functions
of the sets mentioned in the examples.

\begin{itemize}
\item
The convolution of convex-extensible functions
is not necessarily convex-extensible
(Example~\ref{EXicdim2sumhole}).

\item
The convolution of integrally convex functions
is not necessarily integrally convex
(Example~\ref{EXicdim2sumhole}).

\item
The convolution of L$^{\natural}$-convex functions
is not necessarily L$^{\natural}$-convex
(Example~\ref{EXlnatsetMsum}).

\item
The convolution of L-convex functions is not necessarily L-convex
(Example~\ref{EXlsetsum}).

\item
The convolution of multimodular functions is not necessarily multimodular
(Example~\ref{EXmmsetMsum}).

\item
The convolution of globally (resp., locally) discrete midpoint convex functions
is not necessarily globally (resp., locally) discrete midpoint convex
(Example~\ref{EXicdim2sumhole}).
\end{itemize}

The convolution of two L$^{\natural}$-convex functions is integrally convex,
though not L$^{\natural}$-convex.

\begin{proposition} [{\cite[Theorem 8.42]{Mdcasiam}}] \label{PRlfnconvol}
The convolution of two L$^{\natural}$-convex functions is integrally convex.
In particular,
the convolution of two L-convex functions is integrally convex.
\finbox
\end{proposition}

The convolution of three L$^{\natural}$-convex functions
is no longer integrally convex, as the following example shows.

\begin{example} \rm \label{EXminkow3lnatfn}
Let $f_{i} = \delta_{S_{i}}$ $(i=1,2,3)$ for the three \Lnat-convex sets 
$S_1 = \{(0, 0, 0), (1, 1, 0)\}$, 
$S_2 = \{(0, 0, 0), (0, 1, 1)\}$, 
$S_3 = \{(0, 0, 0), (1, 0, 1)\}$
in Example~\ref{EXminkow3lnatset}.
Since
$S_{1}+ S_{2}+ S_{3}
 = \{(0,0,0),(0,1,1),(1,1,0),(1,0,1),(2,1,1),(1,1,2),
 \allowbreak
(1,2,1),(2,2,2)\}$ 
is not an integrally convex set,
$f_{1} \conv f_{2} \conv f_{3} = \delta_{S_{1}+ S_{2}+ S_{3}}$
is not an integrally convex function.
\finbox
\end{example}

We next consider the convolution of a discrete convex function $f$
with a separable convex function $\varphi$.

\begin{proposition} \label{PRsepfnconvol}
\quad

\noindent
{\rm (1)} 
The convolution of an integrally convex function
and a separable convex function is integrally convex.

\noindent
{\rm (2)} 
The convolution of an L$^{\natural}$-convex function
and a separable convex function is L$^{\natural}$-convex.


\noindent
{\rm (3)} 
The convolution of an L-convex function
and a separable convex function is L-convex.


\noindent
{\rm (4)} 
The convolution of an M$^{\natural}$-convex function
and a separable convex function is M$^{\natural}$-convex.

\noindent
{\rm (5)} 
The convolution of a jump \Mnat-convex function
and a separable convex function is jump \Mnat-convex.
\finbox
\end{proposition}

\begin{remark} \rm \label{RMsepfnconvolbib}
Here is a supplement to 
Proposition \ref{PRsepfnconvol} about convolution
with a separable convex function.
Part (1) for integrally convex functions
is due to 
\cite[Theorem 4.2]{MM19projcnvl}.
Parts (2) and (3) for \Lnat-convex and L-convex functions
are given in 
\cite[Theorem 7.11]{Mdcasiam} and \cite[Theorem 7.10]{Mdcasiam}, respectively.
Part (4) for \Mnat-convex functions
is a special case of Theorem \ref{THmfnconvol} (2).
Part (5) for jump \Mnat-convex functions is a 
special case of Theorem \ref{THmfnconvol} (4).
\finbox
\end{remark}

In this connection we note the following facts.

\begin{itemize}
\item
The convolution of an M-convex function $f$
and a separable convex function $\varphi$ is
not necessarily M-convex.
This is simply because 
$\dom \varphi$ is an integer box and hence
$\dom ( f \conv \varphi) = \dom  f + \dom \varphi$
does not lie on a hyperplane of a constant component sum.
Let $\hat \varphi$ denote the restriction of $\varphi$ 
to an M-convex set $S$, that is,
$\hat \varphi(x) = \varphi(x)$ for $x \in S$ and
$\hat \varphi(x) = +\infty$ for $x \not\in S$.
Then $\hat \varphi$ is an M-convex function, and therefore, 
$f \conv \hat \varphi$ is M-convex  by Theorem \ref{THmfnconvol} (3).

\item
The convolution of a jump M-convex function $f$
and a separable convex function $\varphi$ is
not necessarily jump M-convex.
This is simply because 
$\dom \varphi$ is an integer box and hence
$\dom ( f \conv \varphi) = \dom  f + \dom \varphi$
is not a constant-parity jump system.
Let $\hat \varphi$ denote the restriction of $\varphi$ 
to a constant-parity jump system $S$, that is,
$\hat \varphi(x) = \varphi(x)$ for $x \in S$ and
$\hat \varphi(x) = +\infty$ for $x \not\in S$.
Then $\hat \varphi$ is a jump M-convex function, and therefore, 
$f \conv \hat \varphi$ is jump M-convex by 
Theorem \ref{THmfnconvol} (5).

\item
The convolution of a multimodular function
and a separable convex function is 
not necessarily multimodular
(Example~\ref{EXmmsetMsum}).

\item
The convolution of a globally discrete midpoint convex function
and a separable convex function is 
not necessarily discrete midpoint convex.
See Examples \ref{EXdicdim3indic} and \ref{EXdicdim3fn}.

\item
The convolution of a locally discrete midpoint convex function
and a separable convex function is 
not necessarily locally discrete midpoint convex.
See Examples \ref{EXdicdim3indic} and \ref{EXdicdim3fn}.
\end{itemize}

For the convolution of a (globally or locally) discrete midpoint convex function
and a separable convex function,
we show two examples.
In the latter example,  the effective domain of $f$ is an integer box (the unit cube).

\begin{example} \rm \label{EXdicdim3indic}
Let $f = \delta_{S}$
and
$\varphi = \delta_{B}$ for  
$S = \{ (0,0,1), (1,1,0) \}$
and
$B = \{ (0,0,0), (1,0,0) \}$
considered in Example~\ref{EXdicdim3set}.
Function $f$ is (globally and locally) discrete midpoint convex,
and $\varphi$ is separable convex.
The convolution $f \conv \varphi$
coincides with the indicator function of the set 
$S+B = \{ (0,0,1), (1,1,0),  (1,0,1), (2,1,0) \}$.
Since $S+B$ is not discrete midpoint convex,
as shown in Example~\ref{EXdicdim3set},
the function $f \conv \varphi$ is not (globally or locally) 
discrete midpoint convex.
\finbox
\end{example}

\begin{example}[{\cite[Example 4.4]{MM19projcnvl}}]  \rm \label{EXdicdim3fn}
Let $S = \{ (0,0,1), (1,1,0) \}$,
$B = \{ (0,0,0), (1,0,0) \}$, and
$\varphi= \delta_B$, and define
$f: \ZZ^3 \to \RR \cup \{ +\infty\}$
with  
$\dom f = [ (0,0,0), (1,1,1) ]_{\ZZ}$
by
\[
f(x)  =
   \left\{  \begin{array}{ll}
    0            &   (x \in S) ,     \\
    1       &   (x \in [ (0,0,0), (1,1,1) ]_{\ZZ} \setminus S) . \\
                      \end{array}  \right.
\]
This function is (globally and locally) discrete midpoint convex,
while $\varphi$ is separable convex.
The convolution $f \conv \varphi$
is given by 
\[
(f \conv \varphi)(x)  =
   \left\{  \begin{array}{ll}
    0            &   (x \in S+B) ,     \\
    1       &   (x \in [  (0,0,0), (2,1,1) ]_{\ZZ} \setminus (S+B))  \\
                      \end{array}  \right.
\]
with $\dom (f \conv \varphi) = [ (0,0,0), (2,1,1) ]_{\ZZ}$.
For $x = (0,0,1)$, $y = (2,1,0)$ we have 
$\| x - y \|_{\infty} = 2$, \
$x, y \in S+B$, \ 
$\left\lceil (x+y)/2 \right\rceil  = (1,1,1) \not\in S+B$, \
$\left\lfloor (x+y)/2 \right\rfloor = (1,0,0) \not\in S+B$,
and
\[
(f \conv \varphi)(x) = (f \conv \varphi)(y) =0,
\quad
(f \conv \varphi) \left(
\left\lceil \frac{x+y}{2} \right\rceil 
\right)
= 
(f \conv \varphi) \left(
\left\lfloor \frac{x+y}{2} \right\rfloor
\right) = 1 .
\]
Hence 
 $f \conv \varphi$ is not (globally or locally) discrete midpoint convex.
\finbox
\end{example}

\begin{remark} \rm \label{RMprojFromConvol}
The convolution operation with a separable convex function
contains the projection operation as a special case.
Let
$ g(y)  =  \inf \{ f(y,z) \mid z \in \ZZ\sp{N \setminus U} \}$
be the projection of $f$ to $U$
as defined in \eqref{projfndef}.
Let $\varphi$ be the indicator function of the cylinder
$B = \{ (y,z) \in \ZZ\sp{U} \times \ZZ\sp{N \setminus U} \mid y = \veczero_{U} \}$,
where $B$ is an integer box and hence $\varphi$ is a separable convex function.
The convolution $f \conv \varphi$ is given as
\begin{align*}
(f \conv \varphi)(y,z) &= 
\inf \{ f(y',z') + \varphi(y'',z'')  \mid (y,z) = (y',z') + (y'',z'') \}
\\ & =
\inf \{ f(y',z') \mid  (y,z) = (y',z') + (\veczero,z'') \}
\\ & =
\inf \{ f(y,z - z'') \mid z'' \in \ZZ\sp{N \setminus U} \} 
\\ & =
\inf \{ f(y, z') \mid z' \in \ZZ\sp{N \setminus U} \} 
\\ & =
g(y).
\end{align*}
Thus, the value of  projection $g(y)$
is equal to that of convolution $(f \conv \varphi)(y,z)$ for any $z$. 
In this sense the projection can be regarded as a special case of the convolution
with a separable convex function.
\finbox
\end{remark}

\subsection{Integral Legendre--Fenchel transformation}
\label{SCfnconj}

In this section we deal with the operations related to 
integral conjugacy for integer-valued discrete convex functions.
For an integer-valued function 
$f: \ZZ\sp{n} \to \ZZ \cup \{ +\infty \}$,
we define 
a function
$f\sp{\bullet}$ on $\ZZ\sp{n}$ 
by
\begin{align}
f\sp{\bullet}(p)  &= \sup\{  \langle p, x \rangle - f(x)   \mid x \in \ZZ\sp{n} \}
\qquad ( p \in \ZZ\sp{n}),
 \label{conjvexZpZ} 
\end{align}
where 
$\langle p, x \rangle = \sum_{i=1}\sp{n} p_{i} x_{i}$
is the inner product of 
$p=(p_{1}, p_{2}, \ldots, p_{n})$ and 
$x=(x_{1}, x_{2}, \allowbreak \ldots, \allowbreak  x_{n})$.
This function $f\sp{\bullet}$ is referred to as the  
{\em integral conjugate} of $f$,
or the {\em integral Legendre--Fenchel transform} of $f$.
The function $f\sp{\bullet}$ takes values 
in $\ZZ \cup \{ +\infty \}$,
since
$\langle p, x \rangle$ and $f(x)$ are integers (or $+\infty$) for all 
$p \in \ZZ\sp{n}$ and $x \in \ZZ\sp{n}$ and $f(x)$ is finite for some $x$ by the assumption
of $\dom f \not= \emptyset$.
That is, we have 
$f\sp{\bullet}: \ZZ\sp{n} \to \ZZ \cup \{ +\infty \}$.
This allows us to apply the transformation (\ref{conjvexZpZ}) to 
$f\sp{\bullet}$ to obtain
$f\sp{\bullet\bullet} = (f\sp{\bullet})\sp{\bullet}$.
This function $f\sp{\bullet\bullet}$
is called the {\em integral biconjugate} of $f$.

Concerning conjugacy and biconjugacy 
it is natural to ask the following questions
for a given class of discrete convex functions.

\begin{itemize}
\item
For an integer-valued function $f$ in the class,
does the integral conjugate $f\sp{\bullet}$
belong to the same class?
If not, how is it characterized?

\item
For an integer-valued function $f$ in the class,
does integral biconjugacy $f\sp{\bullet\bullet} =f$ hold?
\end{itemize}

For biconjugacy the following theorem has recently been obtained.

\begin{theorem}[\cite{MT18subgrIC}] \label{THbiconjfnZ}
For an integer-valued integrally convex function $f$, 
the integral biconjugate 
$f\sp{\bullet\bullet}$ 
coincides with $f$ itself, i.e.,
$f\sp{\bullet\bullet} = f$.
\finbox
\end{theorem}

The biconjugacy for other classes of discrete convex functions
can be obtained immediately from this result
by the inclusion relations given in Theorem \ref{THfnclassinclusion}.
See also Table~\ref{TB5propertydcfnZ}.

\begin{corollary}  \label{CObiconjfnZ}
Let $f$ be an integer-valued function on $\ZZ\sp{n}$
and $f\sp{\bullet\bullet}$ be its integral biconjugate.

\noindent
{\rm (1)}
If $f$ is separable convex, then $f\sp{\bullet\bullet} = f$.

\noindent
{\rm (2)}
If $f$ is \Lnat-convex, then $f\sp{\bullet\bullet} = f$.


\noindent
{\rm (3)}
If $f$ is L-convex, then $f\sp{\bullet\bullet} = f$.


\noindent
{\rm (4)}
If $f$ is \Mnat-convex, then $f\sp{\bullet\bullet} = f$.


\noindent
{\rm (5)}
If $f$ is M-convex, then $f\sp{\bullet\bullet} = f$.


\noindent
{\rm (6)}
If $f$ is multimodular, then $f\sp{\bullet\bullet} = f$.

\noindent
{\rm (7)}
If $f$ is globally discrete midpoint  convex, then $f\sp{\bullet\bullet} = f$.

\noindent
{\rm (8)}
If $f$ is locally discrete midpoint  convex, then $f\sp{\bullet\bullet} = f$.
\finbox
\end{corollary}

\begin{remark} \rm \label{RMfnbiconjbib}
Here is a supplement to Corollary \ref{CObiconjfnZ} about biconjugacy.
The biconjugacy of L-convex functions in Part (3) was established 
in \cite[Theorem 4.22]{Mdca98},
and that of M-convex functions in Part (5) was in \cite[Theorem 4.8]{Mdca98}.
The biconjugacy of \Lnat- and \Mnat-convex functions 
can be derived easily from these results,
and an explicit statement is made in \cite[Theorem 8.12]{Mdcasiam}.
Part (1) for separable convex functions is a special case of
Part (2) for \Lnat-convex functions and Part (4) for \Mnat-convex functions.
Parts (6), (7), and (8) 
for multimodular functions and 
globally and locally discrete midpoint convex functions
have not been given explicitly in the literature.
\finbox
\end{remark}

\begin{remark} \rm \label{RMbiconjtech}
Here is a technical remark about integral biconjugacy.
For a point $x \in \domZ f$, 
the {\em subdifferential} of $f$ at $x$
is a set of real vectors defined as 
\begin{equation} \label{subgRZdef}
 \subg f(x)
= \{ p \in  \RR\sp{n} \mid    
  f(y) - f(x)  \geq  \langle p, y - x \rangle     \ \mbox{for all }  y \in \ZZ\sp{n} \} ,
\end{equation}
and an element of $\subg f(x)$ is called a {\em subgradient} of $f$ at $x$ \cite{Roc70}.
The condition $\subg f(x) \cap \ZZ\sp{n} \not= \emptyset$ is sometimes referred to as
the {\em integral subdifferentiability}  of $f$ at $x$.
It is known that, for each $x \in \domZ f$,
$f\sp{\bullet\bullet}(x) = f(x)$ holds if and only if $\subg f(x)  \cap \ZZ\sp{n} \not= \emptyset$.
Therefore,
the integral biconjugacy $f\sp{\bullet\bullet} = f$
is equivalent to the integral subdifferentiability of $f$
(under some additional conditions to guarantee
$\dom f\sp{\bullet\bullet} = \dom f$).
See \cite[Lemma 4.2]{Mdca98} as well as \cite[Lemma 2]{MT18subgrIC} for details.
\finbox
\end{remark}

The following example demonstrates the necessity of the assumption of integral convexity
in Theorem \ref{THbiconjfnZ}.

\begin{example}[{\cite[Example 1.1]{Mdca98}}; also \cite{MT18subgrIC}]  \rm \label{EXla1}
Let 
$S = \{ (0,0,0), \pm (1,1,0),  \allowbreak \pm (0,1,1),  \allowbreak  \pm (1,0,1) \}$
and define $f: \ZZ^3 \to \ZZ \cup \{+\infty\}$ by
\begin{align*}
f(x_{1},x_{2},x_{3}) =
 \begin{cases} (x_{1}+x_{2}+x_{3})/2 & (x \in S), \\
                +\infty & (\textrm{otherwise}).  
  \end{cases}
\end{align*}
Note that the function $f$ is indeed integer-valued on $S$.
The set $S$ is hole-free in the sense of (\ref{holefree})
and the function $f$ can be naturally extended to a convex function on 
the convex hull $\overline{S}$ of $S$.
However, this function $f$ is not integrally convex,
since $S$ is not an integrally convex set.
Indeed, for 
$x = [(1,1,0) + (-1,0,-1) ] / 2 =  (0, 1/2, -1/2) \in \overline{S}$, we have
\[
N(x) \cap S = \{ (0,0,0), (0,1,0), (0,0,-1), (0,1,-1)  \} \cap S 
= \{ (0,0,0) \},
\]
and hence $x \not\in \overline{N(x) \cap S}$.

The integral conjugate of $f$ is given as 
\[
f^{\bullet}(p) = \max \{ 0, |p_{1}+p_{2}-1|, |p_{2}+p_{3}-1|, |p_{3}+p_{1}-1| \}
\qquad
(p \in \ZZ^3) .
\]
For the integral biconjugate 
$f^{\bullet\bullet}(x) 
= \sup \{ \langle p, x \rangle - f^{\bullet}(p) \mid p \in \ZZ^3 \}$
we have 
\[
f^{\bullet\bullet}(\veczero) 
= - \inf_{p \in \ZZ^3} \max \{ 0, |p_{1}+p_{2}-1|, |p_{2}+p_{3}-1|, |p_{3}+p_{1}-1| \} 
= -1.
\]
Therefore we have $f^{\bullet\bullet}(\veczero) = -1 \neq 0 = f(\veczero)$. 
This shows $f^{\bullet\bullet} \neq f$.

The subdifferential of $f$ at $x = \veczero$
can be computed as follows.
Let $p \in \subg f(\veczero)$,
which means (by definition) that
$f(y) - f(\veczero) \ge \langle p, y \rangle$ for all $y \in S$,
that is,
\begin{center}
\begin{tabular}{ccc}
$ 1 \ge  p_{1} + p_{2}$, & $ 1 \ge  p_{2} + p_{3}$, & $ 1 \ge  p_{3} + p_{1}$, \\
$-1 \ge -p_{1} - p_{2}$, & $-1 \ge -p_{2} - p_{3}$, & $-1 \ge -p_{3} - p_{1}$.
\end{tabular}
\end{center}
This system of inequalities admits a unique (non-integral) solution
$(p_{1}, p_{2}, p_{3}) = (1/2, 1/2, 1/2)$.
Hence 
$\subg f(\veczero)  = \{ (1/2, 1/2, 1/2) \} $ and
$\subg f(\veczero) \cap \ZZ\sp{3} = \emptyset$.
\finbox
\end{example}

Concerning  conjugate functions we have the following fundamental results
(cf., Table~\ref{TB5propertydcfnZ}).

\begin{theorem}[{\cite[Theorem 4.24]{Mdca98}}, {\cite[Theorem 8.12]{Mdcasiam}}] \label{THconjfnZ}
Let $f$ be an integer-valued function on $\ZZ\sp{n}$
and
$f\sp{\bullet}$ be the integral conjugate of $f$.


\noindent
{\rm (1)}
If $f$ is separable convex, then 
$f\sp{\bullet}$ is separable convex.

\noindent
{\rm (2)}
If $f$ is \Lnat-convex, then $f\sp{\bullet}$ is \Mnat-convex.


\noindent
{\rm (3)}
If $f$ is L-convex, then $f\sp{\bullet}$ is M-convex.


\noindent
{\rm (4)}
If $f$ is \Mnat-convex, then $f\sp{\bullet}$ is \Lnat-convex.


\noindent
{\rm (5)}
If $f$ is M-convex, then $f\sp{\bullet}$ is L-convex.
\finbox
\end{theorem}

The conjugate of a multimodular function
can be captured through the correspondence 
between multimodularity and \Lnat-convexity.

\begin{proposition} \label{PRmmfnconj}
The conjugate $f\sp{\bullet}$
of a multimodular function $f$
can be represented as 
$f\sp{\bullet}(p) = h(D \sp{\top} p)$ 
with an \Mnat-convex function $h$ and the matrix $D$ in \eqref{matDdef}.
\end{proposition}
\begin{proof}
Since $f$ is multimodular,
the function $g$ defined by $g(y) = f(D y)$ is \Lnat-convex,
and hence $g\sp{\bullet}$ is \Mnat-convex
by Theorem \ref{THconjfnZ} (2).
On the other hand, we have
\begin{align*}
f\sp{\bullet}(p) 
 &= \sup\{  \langle p, x \rangle - f(x)   \mid x \in \ZZ\sp{n} \}
= 
\sup\{  \langle p, Dy  \rangle - f(D y)   \mid y \in \ZZ\sp{n} \}
\\ &= 
\sup\{  \langle D\sp{\top} p, y  \rangle - g(y)   \mid y \in \ZZ\sp{n} \}
= 
g\sp{\bullet}(D\sp{\top} p),
\end{align*}
from which we obtain the claim with $h = g\sp{\bullet}$.  
\end{proof}

The integral conjugate of an integrally convex function is not necessarily integrally convex.
This is shown by the following example,
which is obtained from  \cite[Example 4.15]{MS01rel} with a minor modification.

\begin{example} \rm \label{EXconjIC}
Let $S = \{(1, 1, 0, 0), (0, 1, 1, 0), (1, 0, 1, 0), (0, 0, 0, 1)\}$.
This is obviously an integrally convex set, as it is contained in $\{ 0,1 \}\sp{4}$.
Accordingly, its indicator function $\delta_{S}: \ZZ^{4} \to \{0, + \infty\}$ 
is integrally convex.
The integral conjugate $g = \delta_{S}^\bullet$ is given
(cf., \eqref{conjvexZpZ})
 by
\[
 g(p_{1}, p_{2}, p_{3}, p_{4}) 
 = \max\{p_{1} + p_{2}, \  p_{2} + p_{3},\  p_{1} + p_{3}, \  p_{4}\} \qquad (p \in \ZZ^{4}).
\]
Let $\tilde g$ be the local convex extension of $g$.
For $p =(0,0,0,0)$ and $q=(1,1,1,2)$ 
we have  
\[
\tilde g ((p+q)/2) > (g(p)+ g(q)) /2,
\]
since
$(p+q)/2 =(1/2,1/2,1/2,1) = [(1,0,0,1) + (0,1,0,1) + (0,0,1,1) + (1,1,1,1)]/4$,
$\tilde g ((p+q)/2)  = [g(1,0,0,1) + g(0,1,0,1) + g(0,0,1,1) + g(1,1,1,1)]/4 
= (1+1+1+2)/4 = 5/4$, and $(g(p)+ g(q)) /2 = (0+2)/2 = 1$.
Thus the function $g$ violates the condition 
(\ref{intcnvconddist2}) in Theorem~\ref{THfavtarProp33},
and therefore it is not integrally convex.
\finbox
\end{example}

\begin{remark} \rm \label{RMaddconvol}
In convex analysis (for functions in continuous variables)
the addition and convolution operations are known to be conjugate to each other.
In discrete convex analysis some subtlety arises from discreteness. 
Although
\begin{equation}  \label{f1f2convconj}
(f_{1} \conv f_{2})\sp{\bullet}  = {f_{1}}\sp{\bullet} + {f_{2}}\sp{\bullet} 
\end{equation}
holds for any functions $f_{1}, f_{2}: \ZZ\sp{n} \to \ZZ \cup \{ +\infty \}$,
a similar relation 
\begin{equation}  \label{f1f2sumconj}
(f_{1} + f_{2})\sp{\bullet} = 
{f_{1}}\sp{\bullet} \conv {f_{2}}\sp{\bullet}
\end{equation}
may not be true in general.
The identity \eqref{f1f2sumconj} holds 
if $f_{1}$ and $f_{2}$ are \Mnat-convex or 
if $f_{1}$ and $f_{2}$ are \Lnat-convex.
(The proof of \cite[Theorem 8.36]{Mdcasiam} works also for \Lnat-convex functions.)
\finbox
\end{remark}

\begin{table}
\begin{center}
\caption{Conjugacy operations on discrete convex functions}
\label{TB5propertydcfnZ}

\medskip


\begin{tabular}{l|ccc|l}
 Discrete  & Convex- &  Integral & Conjugate & Reference
\\ 
 \quad convexity & extension &   biconj. &  function & 
\\ \hline
 Separable convex & \YES & \YES & separ.~convex & \cite{Mdca98,Mdcasiam}
\\ 
 Integrally convex & \YES & \YES & --- & \cite{MT18subgrIC}
\\ 
\Lnat-convex &  \YES & \YES & \Mnat-convex  & \cite{Mdcasiam}
\\ 
L-convex  &  \YES & \YES & M-convex  & \cite{Mdca98,Mdcasiam}
\\ 
\Mnat-convex  &  \YES & \YES & \Lnat-convex  & \cite{Mdcasiam}
\\ 
M-convex & \YES & \YES & L-convex  & \cite{Mdca98,Mdcasiam}
\\ 
Multimodular  &  \YES & \YES & $\cong$ \Mnat-convex
\\ 
Globally d.m.c.  &  \YES &  \YES & ---
\\ 
Locally d.m.c.  &  \YES & \YES & ---
\\ 
Jump \Mnat-convex & \NO  & \NO & ---
\\ 
Jump M-convex & \NO  & \NO & ---
\\ \hline
\multicolumn{5}{l}{``Y'' means ``Yes, this property holds for this function class.''} 
\\
\multicolumn{5}{l}{``\textbf{\textit{N}}'' means ``No, this property does not hold for this function class.''} 
\end{tabular}
\end{center}
\end{table}


\section*{Acknowledgement}
The author thanks Kokichi Sugihara and Takashi Tsuchiya for encouragement and help, 
and Satoru Fujishige and Satoko Moriguchi for comments.
This work was supported by 
CREST, JST, Grant Number JPMJCR14D2, Japan, and
JSPS KAKENHI Grant Number 26280004.




\end{document}